\documentclass{amsart}

\usepackage{amsrefs}

\usepackage{amsmath}
\usepackage{amsthm}
\usepackage{amssymb}
\usepackage{tabu}

\usepackage{bbm}

\newcommand{\LO}{\ensuremath{L_{\gamma,\alpha}}}
\newcommand{\Bm}{\ensuremath{\beta_-(\gamma)}}
\newcommand{\Bp}{\ensuremath{\beta_+(\gamma)}}
\newcommand{\R}{\ensuremath{\mathbb{R}}}
\newcommand{\h}{\ensuremath{H^{\frac{\alpha}{2}}_0(\mathbb{R}^n)}}
\newcommand{\homega}{\ensuremath{H_0^{\frac{\alpha}{2}}(\Omega)}}
\newcommand{\N}{\ensuremath{\mathbb{N}}}

\def \rr {\mathbb{R}}
\def \nn {\mathbb{N}}
\def \rn {\mathbb{R}^n}
\def \ue {u_\epsilon}
\def \Te {T_\epsilon}
\def \eps {\epsilon}
\def \crits {2_{\alpha}^\star(s)}
\def \bp {\beta_+}
\def \bm {\beta_-}

\title[Mass and Asymptotics associated to Fractional Hardy-Schr\"odinger Operators]{Mass and Asymptotics associated to Fractional Hardy-Schr\"odinger Operators in Critical Regimes}
\author{Nassif Ghoussoub}
\address{Nassif Ghoussoub, Department of Mathematics, 1984 Mathematics Road, The University of British Columbia, BC, Canada V6T 1Z2}
\email{nassif@math.ubc.ca}

\author{Fr\'ed\'eric Robert}
\address{Fr\'ed\'eric Robert, Institut \'Elie Cartan, Universit\'e de Lorraine, BP 70239, F-54506 Vand{\oe}uvre-l\`es-Nancy, France}
\email{frederic.robert@univ-lorraine.fr}

\author{Shaya Shakerian}
\address{Shaya Shakerian, Department of Mathematics, 1984 Mathematics Road, The University of British Columbia, BC, Canada V6T 1Z2}
\email{shaya@math.ubc.ca}

\author{Mingefeng Zhao}
\address{Mingefeng Zhao, Department of Mathematics, 1984 Mathematics Road, The University of British Columbia, BC, Canada V6T 1Z2}
\email{mingefeng@math.ubc.ca}

\date{April 27th, 2017}

\newtheorem{definition}{Definition}[section]
\newtheorem{theorem}[definition]{Theorem}
\newtheorem{lemma}[definition]{Lemma}
\newtheorem{proposition}[definition]{Proposition}

\newtheorem{remark}[definition]{Remark}
\theoremstyle{definition}

\begin{document}

\begin{abstract} We consider linear and non-linear boundary value problems associated to the fractional Hardy-Schr\"odinger operator $ \LO: = ({-}{ \Delta})^{\frac{\alpha}{2}}-  \frac{\gamma}{|x|^{\alpha}}$ on domains of $\R^n$ containing the singularity $0$, where $0<\alpha<2$ and  $ 0 \le \gamma < \gamma_H(\alpha)$, the latter being the best constant in the fractional Hardy inequality on $\R^n$.  
We tackle the existence of least-energy solutions for the borderline boundary value problem $(\LO-\lambda I)u= {\frac{u^{\crits-1}}{|x|^s}}$ on $\Omega$, where $0\leq s <\alpha <n$ and $ {\crits}={\frac{2(n-s)}{n-{\alpha}}}$ is the critical fractional Sobolev exponent. We show that if $\gamma$ is below a certain threshold $\gamma_{crit}$, then such solutions exist for all $0<\lambda <\lambda_1(\LO)$, the latter being the first eigenvalue of $\LO$. On the other hand, for $\gamma_{crit}<\gamma <\gamma_H(\alpha)$, we prove existence of such solutions only for those $\lambda$ in $(0, \lambda_1(\LO))$ for which the domain $\Omega$ has a positive {\it fractional Hardy-Schr\"odinger mass} $m_{\gamma, \lambda}(\Omega)$. This latter notion is introduced by way of an invariant of the linear equation $(\LO-\lambda I)u=0$ on $\Omega$.
\end{abstract}
 
\maketitle

\section{Introduction} 
We study various linear and non-linear equations involving the fractional Hardy-Schr\"odinger operator $ \LO: = ({-}{ \Delta})^{\frac{\alpha}{2}}-  \frac{\gamma}{|x|^{\alpha}}$, where  $\displaystyle (-\Delta)^{\frac{\alpha}{2}}$ is the so-called fractional Laplacian, defined below. Throughout this paper, we shall assume that  
\begin{equation}\label{cond1}
\hbox{$0<\alpha<n$ \quad and \quad $ 0 \le \gamma < \gamma_H(\alpha)=2^\alpha \frac{\Gamma^2(\frac{n+\alpha}{4})}{\Gamma^2(\frac{n-\alpha}{4})}$},
\end{equation}
the latter being the best constant in the fractional Hardy constant on $\R^n$ (see below). Our main focus will be on the case when $\alpha <2$, that is when $\displaystyle (-\Delta)^{\frac{\alpha}{2}}$ is not a local operator. We shall study problems on bounded domains, but will start by recalling the properties of  $\displaystyle (-\Delta)^{\frac{\alpha}{2}}$ on the whole of $\R^n$, where  it can be defined on the 
Schwartz class $\mathcal{S}$ (the space of rapidly decaying $C^\infty$ functions on $\R^n$) via the Fourier transform,
\begin{equation*}
 (-\Delta)^{\frac{\alpha}{2}}u= \mathcal{F}^{-1}(| 2 \pi \xi|^{\alpha}\mathcal{F}(u)).
 \end{equation*}
Here $ \mathcal{F}(u)$ is the Fourier transform of $u$, $\displaystyle \mathcal{F}(u)(\xi)=\int_{\R^n} e^{-2\pi i x.\xi} u(x) dx$. See Servadei-Valdinoci \cite{Servadei-Valdinoci} and references therein for the basics on the fractional Laplacian. For $\alpha \in (0,2)$, the fractional Sobolev space $\h$ is defined as the completion of $C_c^{\infty}(\R^n)$ under the norm 
$$\|u\|_{\h}^2= \int_{\mathbb{R}^n}|2\pi \xi |^{\alpha} |\mathcal{F}u(\xi)|^2 d\xi =\int_{\mathbb{R}^n} |(-\Delta)^{\frac{\alpha}{4}}u|^2 dx.$$
By Proposition 3.6 in Di Nezza-Palatucci-Valdinoci \cite{Hitchhikers-guide} (see also Frank-Lieb-Seiringer \cite{Frank-Lieb-Seiringer}), the following relation holds: 
For $u \in \h,$ 
\begin{equation*}
 \int_{\mathbb{R}^n}|2\pi \xi |^{\alpha} |\mathcal{F}u(\xi)|^2 d\xi=\frac{C_{n,\alpha}}{2}\int_{\R^n}  \int_{\R^n} \frac{|u(x)- u(y)|^2}{|x-y|^{n+\alpha}} dxdy, 
 \end{equation*} 
 where $C_{n,\alpha}=\frac{2^\alpha\Gamma\left(\frac{n+\alpha}{2}\right)}{\pi^{\frac{n}{2}}\left|\Gamma\left(-\frac{\alpha}{2}\right)\right|}.$

The fractional Hardy inequality in $\R^n$ then states that 
\begin{equation*}
\gamma_H(\alpha):= \inf\left\{\frac{ \int_{\R^n} |({-}{ \Delta})^{\frac{\alpha}{4}}u|^2 dx  }{\int_{\R^n} \frac{|u|^2}{|x|^{\alpha}}dx};\, u \in \h \setminus \{0\}\right\}=2^\alpha \frac{\Gamma^2(\frac{n+\alpha}{4})}{\Gamma^2(\frac{n-\alpha}{4})},
\end{equation*}
which means that the fractional Hardy-Schr\"odinger operator $ \LO$ is positive whenever (\ref{cond1}) is satisfied. 
In this case, a Hardy-Sobolev type inequality holds for $ \LO$. It states that if $ 0 \le s < \alpha < n$, and  $ {\crits}={\frac{2(n-s)}{n-{\alpha}}}$, then $\mu_{\gamma,s,\alpha}(\R^n)$ is finite and strictly positive, where the latter is the best constant 
\begin{equation} \label{def:mu}
\mu_{\gamma,s,\alpha}(\R^n):= \inf\limits_{u \in \h \setminus \{0\}} \frac{ \int_{\R^n} |({-}{ \Delta})^{\frac{\alpha}{4}}u|^2 dx - \gamma \int_{\R^n} \frac{|u|^2}{|x|^{\alpha}}dx }{(\int_{\R^n} \frac{|u|^{\crits}}{|x|^{s}}dx)^\frac{2}{\crits}}.
\end{equation}
Note that any minimizer for (\ref{def:mu}) leads --up to a constant-- to a  
variational solution of the following borderline problem  on $\R^n$,
 \begin{equation}\label{Main:problem}
\left\{\begin{array}{rl}
({-}{ \Delta})^{\frac{\alpha}{2}}u- \gamma \frac{u}{|x|^{\alpha}}= {\frac{u^{\crits-1}}{|x|^s}} & \text{in }  {\R^n}\\
 u\geq 0\; ;\; u\not\equiv 0 \,\,\,\,\,\,\,\,\,\,\,\,\,\,\,\,\,  & \text{in }  \mathbb{R}^n.
\end{array}\right.
\end{equation}
Indeed, a function $ u\in \h$ is said to be a weak solution to (\ref{Main:problem}) if $u\geq 0$, $u\not\equiv 0$ and for any $\displaystyle \varphi\in\h$, we have
\begin{eqnarray*}\label{def:weak:sol}
\frac{C_{n,\alpha}}{2}\int_{\mathbb{R}^n}\int_{\mathbb{R}^n} \frac{(u(x)-u(y))(\varphi(x)-\varphi(y))}{|x-y|^{n+\alpha}}\ dxdy=\int_{\mathbb{R}^n}(\gamma\frac{u}{|x|^\alpha}+\frac{u^{2_\alpha^*(s)-1}}{|x|^s})\varphi\ dx.
\end{eqnarray*}
Unlike the case of the Laplacian ($\alpha =2$), no explicit formula is known for the best constant $\mu_{\gamma,s,\alpha}(\R^n)$ nor for the extremals where it is achieved. We therefore try to describe their asymptotic profile whenever they exist.  
This was considered in  Ghoussoub-Shakerian 
\cite{GS}, where the following is proved.
\begin{theorem}[Ghoussoub-Shakerian \cite{GS}] 
Suppose $0<\alpha<2$, $ 0 \le s < \alpha<n$, and $\gamma <2^\alpha \frac{\Gamma^2(\frac{n+\alpha}{4})}{\Gamma^2(\frac{n-\alpha}{4})}$. 
\begin{enumerate}
\item If either $ s > 0$ or $\{s=0$ and $\gamma \ge 0 \}$, then $\mu_{\gamma,s,\alpha}(\R^n)$ is attained.

\item If $s=0$ and $\gamma < 0$, then there are no extremals for $\mu_{\gamma,s,\alpha}(\R^n).$

\item If either ${0 < \gamma < \gamma_H(\alpha) } \text{ or } \{ \gamma=0 \text{ and } 0<s<\alpha  \},$
 then any non-negative minimizer for $\mu_{\gamma,s,\alpha}(\R^n)$ is positive, radially symmetric, radially decreasing,  
 and approaches zero as ${|x| \to \infty}.$
 \end{enumerate} 
\end{theorem}
Note that the cases when $\gamma=0$ are by now well known. Indeed, it was stated in \cite{Cotsiolis-Tavoularis} that the infimum in $\mu_{0,0,\alpha}(\R^n)$ is attained. Actually, a function $\tilde{u} \in \h \setminus \{0\}$ is an  extremal for $\mu_{0,0,\alpha}(\R^n)$ if and only if there exist $x_0 \in \R^n,$ $k \in \R \setminus \{0\}$ and  $r>0$ and  such that
\begin{equation*}
 \tilde{u}(x) = k \left(  r^2 + |x - x_0|^2 \right)^{-\frac{(n-\alpha)}{2}} \quad  \text{ for all } x \in \R^n.
\end{equation*}
Asymptotic properties of the positive extremals of $\mu_{0,s,\alpha}(\R^n)$ (i.e., when $\gamma=0$ and $0<s<\alpha$) were  given by  Y. Lei  \cite{Lei}, Lu-Zhu \cite{Lu-Zhu}, and  Yang-Yu \cite{Yang-Yu}.
The latter proved that an extremal $\bar{u}(x)$ for $\mu_{0,s,\alpha}(\R^n)$ must have the following behaviour: There is $C>0$ such that 
\begin{equation}\label{est:ext:Lei}
 C^{-1} \left(1 + |x|^2 \right)^{-\frac{(n-\alpha)}{2}} \le  \bar{u}(x) \le C \left(1 + |x|^2 \right)^{-\frac{(n-\alpha)}{2}} \quad \hbox{for all  $x \in \R^n$.}
  \end{equation}
  Recently,  Dipierro-Montoro-Peral-Sciunzi \cite{DMPS} found a similar control of the extremal for $\mu_{\gamma,0,\alpha}(\R^n)$ (i.e., when $0 < \gamma <\gamma_H(\alpha)$ and $s=0$). Our first result is an improvement of their estimate since it gives the exact asymptotic behaviour of the extremal of $\mu_{\gamma,s,\alpha}(\R^n)$ in the general case. For that, we consider the function 
\begin{equation}\label{Definition of Psi.0}
\Psi_{n,\alpha}(\beta) := 2^\alpha \frac{\Gamma(\frac{n-\beta}{2})\Gamma(\frac{\alpha+\beta}{2})}{  \Gamma(\frac{n-\beta-\alpha}{2})\Gamma(\frac{\beta}{2} ) }. 
\end{equation}

\begin{theorem}\label{th:asymp:ext}
Assume $0 \le s<\alpha<2,$ $n> \alpha$ and $\displaystyle 0\leq\gamma<\gamma_H(\alpha)$. Then any positive  extremal $\displaystyle u\in \h$ for $\mu_{\gamma,s,\alpha}(\R^n)$ satisfies $u\in C^1(\rn\setminus\{0\})$ and  
\begin{equation}\label{ass1}
\lim_{x\to 0}|x|^{\Bm}u(x)=\lambda_0\hbox{ and }\lim_{|x|\to \infty}|x|^{\Bp}u(x)=\lambda_\infty,
\end{equation}
where $\lambda_0,\lambda_\infty>0$ and $\Bm$ (resp., $\Bp$) is the unique solution in 
$\left(0,\frac{n-\alpha}{2}\right)$ (resp., in $\left(\frac{n-\alpha}{2},n-\alpha\right))$ of the equation $\Psi_{n,\alpha}(t)=\gamma$. In particular, there exists $C_1,C_2>0$ such that 
$$\frac{C_1}{|x|^{\Bm} + |x|^{\Bp}} \leq u(x)\leq \frac{C_2}{|x|^{\Bm} + |x|^{\Bp}} \quad \hbox{for all $x\in \rn\setminus\{0\}$.}$$

\end{theorem}
\begin{remark} \rm Note that a direct consequence of Theorem  \ref{th:asymp:ext} is (\ref{est:ext:Lei})  and the corresponding control by Dipierro-Montoro-Peral-Sciunzi \cite{DMPS}.

\medskip\noindent Also note that if $\alpha = 2$, that is when  the fractional Laplacian is  the classical Laplacian, the best constant in the Hardy inequality is then $\gamma_H(2)= \frac{(n-2)^2}{4}$. The best constant associated with the Hardy-Sobolev inequality is
\begin{equation*} 
\mu_{\gamma,s,2}(\R^n):= \inf\limits_{u \in D^{1,2} (\R^n) \setminus \{0\}} \frac{ \int_{\R^n} |\nabla u|^2 dx - \gamma \int_{\R^n} \frac{|u|^2}{|x|^2}dx }{(\int_{\R^n} \frac{|u|^{2^\star(s)}}{|x|^{s}}dx)^\frac{2}{2^*(s)}},
\end{equation*}
where  $s \in [0,2),$ $ 2^\star(s):=\frac{2(n-s)}{n-2},$ $ 0 \le \gamma < \gamma_H(2)= \frac{(n-2)^2}{4}$ and $D^{1,2}(\R^n)$ is the completion of $C_c^\infty(\Omega)$ with respect to the norm $ \|u\|^2 = \int_{\R^n} |\nabla u|^2dx.$ 
The extremals for $\mu_{\gamma,s,2}(\R^n) $ are then explicit and are  given by multiples of the functions $ u_\epsilon (x) =  \epsilon^{-\frac{n-2}{2}} U(\frac{x}{\epsilon})$ for  $\epsilon>0$, where
\begin{equation*}
U(x) = \frac{1}{\left(|x|^\frac{(2-s) \sigma_-(\gamma) }{n-2} + |x|^{\frac{(2-s) \sigma_+(\gamma)}{n-2}}\right)^{\frac{n-2}{2-s}}} \quad \text{ for } \R^n \setminus \{0\},
\end{equation*}
and 
\begin{equation*}
\sigma_\pm(\gamma) = \frac{ n-2}{2} \pm\sqrt{\frac{(n-2)^2}{4} - \gamma}.
\end{equation*}
Note that the radial function $u(x)= |x|^{-\beta}$ is a solution of $ L_{\gamma,2} (u)=0$ on $\R^n \setminus \{ 0 \}$ if and only if $\beta \in \{\sigma_-(\gamma), \sigma_+(\gamma)\}.$
\end{remark}

\medskip\noindent Back to the case $0<\alpha<2$, we now turn to when $\Omega$ is a smooth bounded domain in $\R^n$ with $0$ in its interior. The best constant in the corresponding fractional Hardy-Sobolev inequality is then,
\begin{equation*} 
\mu_{\gamma,s,\alpha}(\Omega):= \inf\limits_{u \in \homega\setminus \{0\}} \frac{ \frac{C_{n,\alpha}}{2} \int_{\R^n}  \int_{\R^n} \frac{|u(x)- u(y)|^2}{|x-y|^{n+\alpha}} dxdy - \gamma \int_{\Omega} \frac{|u|^2}{|x|^{\alpha}}dx }{(\int_{\Omega} \frac{|u|^{\crits}}{|x|^{s}}dx)^\frac{2}{\crits}},
\end{equation*}
where $\homega$ is  the closure of $C^\infty_c(\Omega)$ with respect to the norm  
$$\|u\|^2_{\homega} = \frac{C_{n,\alpha}}{2} \int_{\R^n}  \int_{\R^n} \frac{|u(x)- u(y)|^2}{|x-y|^{n+\alpha}} dxdy = \int_{\R^n} |(-\Delta)^{\frac{\alpha}{4}}u|^2 dx.$$ 
In Proposition \ref{prop:val:best:cst}, we note that --just like the case when $\alpha=2$-- we have $ \mu_{\gamma,s,\alpha}(\Omega)=\mu_{\gamma,s,\alpha}(\R^n)$, 
and therefore  (\ref{Main:problem}) restricted to $\Omega$, with Dirichlet boundary condition has no extremal, unless $\Omega$ is essentially $\R^n$. 
We therefore resort to a setting popularized by Brezis-Nirenberg \cite{Brezis-Nirenberg} by considering the following boundary value problem:
\begin{equation}\label{syst:BN}
\left\{\begin{array}{rl}
\displaystyle ({-}{ \Delta})^{\frac{\alpha}{2}}u- \gamma \frac{u}{|x|^{\alpha}}= {\frac{u^{\crits-1}}{|x|^s}}+ \lambda u & \text{in }  {\Omega}\vspace{0.1cm}\\
u\geq 0 \,\,\,\,\,\,\,\,\,\,\,\,\,\,\,\,\,\,\,\,\,\,\,\,\,\,\,\,\,\,\,\,\,\,  & \text{in } \Omega ,\\
 u=0 \,\,\,\,\,\,\,\,\,\,\,\,\,\,\,\,\,\,\,\,\,\,\,\,\,\,\,\,\,\,\,\,\,\,  & \text{in } \R^n \setminus \Omega
\end{array}\right.
\end{equation}
where $0< \lambda < \lambda_1(\LO)$ and  $\lambda_1(\LO)$ is the first eigenvalue of the operator $ \LO=(-\Delta)^{\frac{\alpha}{2}}-  \frac{\gamma}{|x|^\alpha}$ with Dirichlet boundary condition, that is,
$$\lambda_1 := \lambda_1(\LO) =\inf\limits_{u \in \homega \setminus \{0\}} \frac{ \frac{C_{n,\alpha}}{2} \int_{\R^n} \int_{\R^n} \frac{|u(x) - u(y) |^2}{|x-y|^{n+\alpha}} dx dy- \gamma \int_\Omega \frac{u^2}{|x|^\alpha}}{\int_{\Omega} u^2 dx}. $$
One then considers the quantity
\begin{align*}
\mu_{\gamma,s,\alpha,\lambda}(\Omega) &= \inf\limits_{u \in \homega \setminus \{0\}} \frac{\frac{C_{n,\alpha}}{2} \int_{\R^n} \int_{\R^n} \frac{|u(x) - u(y) |^2}{|x-y|^{n+\alpha}} dx dy -\gamma \int_{\Omega} \frac{u^2}{|x|^{\alpha}}dx - \lambda \int_{\Omega} u^2 dx }{\left(\int_{\Omega} \frac{u^{\crits}}{|x|^s}dx\right)^{\frac{2}{\crits}}},
\end{align*}
and uses the fact that compactness is restored as long as  $\mu_{\gamma,s,\alpha,\lambda}(\Omega) < \mu_{\gamma,s,\alpha}(\R^n)$; see Proposition \ref{prop:exist} and also \cite{Brezis-Nirenberg} for more details. This type of condition is now classical in borderline variational problems; see Aubin \cite{Aubin} and Brezis-Nirenberg \cite{Brezis-Nirenberg}. 

\medskip\noindent When $\alpha=2$, i.e., in the case of the standard Laplacian, the minimization problem $\mu_{\gamma,s,\alpha,\lambda}(\Omega)$ has been extensively studied, see for example Lieb \cite{Lieb}, Chern-Lin \cite{Chern-Lin}, Ghoussoub-Moradifam \cite{Ghoussoub-Moradifam} and Ghoussoub-Robert \cite{Ghoussoub-Robert-2014}. The non-local case has also been the subject of several studies, but in the absence of the Hardy term, i.e., when $\gamma=0$. 
In \cite{Servadei}, Servadei proved the existence of extremals for $\mu_{0,0,\alpha,\lambda}(\R^n),$ and completed the study of problem (\ref{syst:BN}) which has been initiated by Servadei-Valdinoci \cites{Servadei-Valdinoci-low-dimension,Servadei-Valdinoci}. Recently, it has been shown by Yang-Yu \cite{Yang-Yu} that there exists a positive extremal for $\mu_{0,s,\alpha,\lambda}(\R^n)$ when $s \in [0,2).$ In this paper, we consider 
the remaining cases. 

\medskip\noindent In the spirit of Jannelli \cite{Jannelli}, who dealt with the Laplacian case, we observe that problem (\ref{syst:BN}) is deeply influenced by the value of the parameter $\gamma$.  Roughly speaking, if $\gamma$ is sufficiently small then $\mu_{\gamma,s,\alpha,\lambda}(\Omega)$ is attained for any $0<\lambda <\lambda_1$. This is essentially what was obtained by Servadei-Valdinoci \cite{Servadei-Valdinoci} when $s=\gamma=0$ and $n\geq 2\alpha$ via local arguments. This is, however not the case, when $\gamma$ is closer to $\gamma_H(\alpha)$, which amounts to dealing with low dimensions: see for instance Servadei-Valdinoci \cite{Servadei-Valdinoci-low-dimension}. In this context of low dimension, the local arguments generally fail, and it is necessary to use \emph{global} arguments via the introduction of a notion of mass in the spirit of Schoen \cite{schoen}. In the present case, and as in the work of Ghoussoub-Robert \cite{Ghoussoub-Robert-2015}, we define a notion of mass for the operator $\LO-\lambda I$, which again turns out to be critical for this non-local case. The mass is defined via the following key result. 

\begin{theorem}\label{Theorem:TheMass} Let $\Omega$ be a bounded smooth domain in $\mathbb{R}^n$ ($n > \alpha$) and consider, for $0<\alpha<2$, the boundary value problem
 \begin{equation}\label{Main problem:TheMass}
\left\{\begin{array}{rl}
(-\Delta)^{\frac{\alpha}{2}}H-\left( \frac{\gamma}{|x|^\alpha}+a(x)\right) H=0 & \text{in }  {\Omega \setminus \{0\}}\\
 H\geq 0  & \text{in }  \Omega \setminus \{0\}\\
 H=0 & \text{in } \R^n \setminus \Omega,
\end{array}\right.
\end{equation}
where $a(x)\in C^{0,\tau}(\overline{\Omega})$ for some $\tau\in(0,1)$. Assuming  the operator $(-\Delta)^{\frac{\alpha}{2}}-(\frac{\gamma}{|x|^\alpha}+a(x))$ coercive, there exists then a threshold $-\infty< \gamma_{crit}(\alpha)<\gamma_H(\alpha)$ such that for any $\gamma$ with $\gamma_{crit}(\alpha)<\gamma<\gamma_H(\alpha)$, 
there exists a unique solution to (\ref{Main problem:TheMass}) (in the sense of Definition \ref{def:punct}) $H :\Omega\to \rr$, $H\not\equiv 0$, and a constant $c \in \R$ such that 
$$H(x) = \frac{1}{|x|^{\Bp}}+ \frac{c}{|x|^{\Bm}}+ o \left( \frac{1}{|x|^{\Bm}} \right) \quad \text{ as } x \to 0.$$ 

\smallskip\noindent We define the \emph{fractional Hardy-singular internal mass} of $\Omega$ associated to the operator $\LO  $ to be $$m^\alpha_{\gamma,a}(\Omega): =c \in \R.$$
\end{theorem}
\noindent We then prove the following existence result, which complements
those in \cite{Servadei} and \cite{Yang-Yu} to the case when $\gamma>0$. 
\begin{theorem}\label{th:1}
Let $\Omega$ be a smooth bounded domain in $\R^n (n > \alpha)$ such that $0 \in \Omega,$ and let  $0 \le s < \alpha,$ $0 \le \gamma < \gamma_H(\alpha)$.Then, 
there exist extremals for $\mu_{\gamma,s,\alpha,\lambda}(\Omega)$  under one of the following two conditions:
\begin{enumerate}
\item $0 \le \gamma \le \gamma_{crit}(\alpha)$ and $0<\lambda<\lambda_1(\LO)$,
 \item $\gamma_{crit}(\alpha)<\gamma <\gamma_H(\alpha)$, $0<\lambda<\lambda_1(\LO)$ and  $m^\alpha_{\gamma,\lambda}(\Omega)>0$.
\end{enumerate}
\end{theorem}
The idea of studying how critical behavior occurs while varying a parameter $\gamma$ on which an operator $\LO$ continuously depends goes back to \cite{Jannelli}, who considered the classical Hardy-Schr\"odinger operator $ L_{\gamma,2}:= -\Delta -\frac{\gamma}{|x|^2}$, and showed the existence of extremals for any $\lambda >0$ provided $0\leq \gamma \leq  \frac{(n-2)^2}{4}-1$. In this case, $\gamma_{crit}(2)= \frac{(n-2)^2}{4}-1$. The definition of the mass and the counterpart of Theorem \ref{th:1} for the operator $ L_{\gamma,2}$ was established by Ghoussoub-Robert \cite{Ghoussoub-Robert-2015}.  The complete picture can be described as follows. \\

\begin{center}
\begin{tabular}{ | m{3.3cm} | m{1.8 cm}| m{1.5cm} | m{2.5cm}|  m{1.5cm}|}
\hline
 Hardy term & Dimension & Singularity & Analytic. cond. & Extremals\\
\hline
$ \quad  0 \le  \gamma \le \gamma_{crit}(\alpha)$ & $n \ge 2 \alpha $ & $s\geq 0$ & \qquad $\lambda >0$ & Yes\\ 
\hline
$\gamma_{crit}(\alpha)< \gamma < \gamma_H(\alpha)$& $n \ge 2 \alpha$ & $s\geq 0$& \quad $m^\alpha_{\gamma,\lambda}(\Omega)>0$ & Yes\\ 
\hline
$0 \le  \gamma < \gamma_H(\alpha)$& $ \alpha < n <  2 \alpha$ & $s\geq 0$& \quad $m^\alpha_{\gamma,\lambda}(\Omega)>0$ & Yes\\ 
\hline\\
\end{tabular}
\end{center}

\medskip\noindent Even though the constructions and the methods are heavily inspired by the work of Ghoussoub-Robert \cite{Ghoussoub-Robert-2015} on the Laplacian case, the fact that the operator is nonlocal here induces several fundamental difficulties that had to be overcome. First, the construction of the mass in the local case uses a precise classification of singularities for solutions of corresponding elliptic equations, that follows from the comparison principle stating that behavior in a domain is governed by the behavior on its boundary. In the nonlocal case, this fails since one needs to consider the whole complement of the domain, and not only its boundary. We were able to bypass this difficulty by using sharp regularity results available for the fractional Laplacian. Another difficulty we had to face came from the test-functions estimates in the presence of the mass. In the classical local case, one estimates the associated functional on a singular test-function, counting on the mass to appear after suitable integrations by parts. In the nonlocal context, this strategy fails. We overcome this difficulty by looking at the integral on the boundary of a domain as a limit of integrals on the domain after multiplying by a cut-off functions whose support converge to the boundary. This process is well-defined in the nonlocal context and proves to be efficient in tackling the estimates involving the mass.

\section{The fractional Hardy-Schr\"odinger operator \LO  \ on $\R^n$ }
In this section, we study the local behavior of solutions of the fractional Hardy-Schr\"odinger operator $ \LO:= (-\Delta)^{\frac{\alpha}{2}}-\frac{\gamma}{|x|^\alpha} $ on  $\R^n$. The most basic solutions for $\LO u = 0$ on $\R^n$ are of the form $u(x) = |x|^{- \beta}$, and a straightforward computation  yields  (see \cite{Frank-Lieb-Seiringer})
$$(-\Delta)^\frac{\alpha}{2} |x|^{-\beta} = \Psi_{n,\alpha}(\beta)  |x|^{-\beta-\alpha}  \text{ in the sense of } \mathcal{S}'(\R^n)\hbox{ when }0<\beta<n-\alpha, $$ 
where
\begin{equation}\label{Definition of Psi}
\Psi_{n,\alpha}(\beta) := 2^\alpha \frac{\Gamma(\frac{n-\beta}{2})\Gamma(\frac{\alpha+\beta}{2})}{  \Gamma(\frac{n-\beta-\alpha}{2})\Gamma(\frac{\beta}{2} ) } . 
\end{equation}
Recall that  the best constant in the fractional Hardy inequality
\begin{equation*} 
\gamma_H(\alpha):=\mu_{0,\alpha,\alpha}(\R^n) = \inf\left\{\frac{ \int_{\R^n} |({-}{ \Delta})^{\frac{\alpha}{4}}u|^2 dx}{ \int_{\R^n}\frac{|u|^2}{|x|^{\alpha}} dx}; \,\,  u \in \h \setminus \{0\}\right\}
\end{equation*}
 is never achieved (see Fall \cite{Fall}), is equal to  $\Psi_{n,\alpha}(\frac{n-\alpha}{2})=2^\alpha \frac{\Gamma^2(\frac{n+\alpha}{4})}{\Gamma^2(\frac{n-\alpha}{4})}$ (see Herbst and Yafaev \cites{Herbst,Yafaev}), and it converges to the best classical Hardy constant $\gamma_H(2)=\frac{(n-2)^2}{4}$ whenever ${\alpha \to 2}$. 

\medskip\noindent We summarize some properties of the function $\beta\mapsto\Psi_{n,\alpha}(\beta)$ which will  be used freely in this section. They are essentially consequences from known properties of Gamma function $\Gamma$.

\begin{proposition}[Frank-Lieb-Seiringer \cite{Frank-Lieb-Seiringer}]\label{rem:M-1} The following properties hold:

\begin{enumerate}
\item $\Psi_{n,\alpha}(\beta)>0$ for all $\beta\in(0,n-\alpha)$.

\item The graph of $\Psi_{n,\alpha}$ in $(0,n-\alpha)$ is symmetric with respect to $ \frac{n-\alpha}{2}$, that is, 
$$\Psi_{n,\alpha}(\beta)=\Psi_{n,\alpha}(n-\alpha-\beta) \hbox{ for all $\beta\in(0,n-\alpha)$.}
$$

\item $\Psi_{n,\alpha}$ is strictly increasing in $(0,\frac{n-\alpha}{2})$, and  strictly decreasing in $ (\frac{n-\alpha}{2},n-\alpha)$.

\item $\displaystyle \Psi_{n,\alpha}\left(\frac{n-\alpha}{2}\right)=\gamma_H(\alpha).$

\item $\displaystyle \lim_{\beta\searrow0}\ \Psi_{n,\alpha}(\beta)=\lim_{\beta\nearrow n-\alpha}\ \Psi_{n,\alpha}(\beta)=0$.

\item For any $\gamma\in(0,\gamma_H(\alpha))$, there exists a unique $\Bm \in (0,\frac{n-\alpha}{2})$ such that $\Psi_{n,\alpha}(\Bm)=\gamma$.

\item For any $0<\beta\leq n-\alpha$, we have that
\begin{equation}\label{lap:x:theta:dist}
(-\Delta)^{\frac{\alpha}{2}}|x|^{-\beta}=\Psi_{n,\alpha}(\beta)|x|^{-\alpha-\beta}+c_{n,\alpha}{\bf 1}_{\{\beta=n-\alpha\}}\delta_0\hbox{ in }\mathcal{S}'(\R^n),
\end{equation}
where we define $\Psi_{n,\alpha}(n-\alpha)=0$ and $c_{n,\alpha}>0$ is a constant.
\end{enumerate}

\end{proposition}
In particular, for $0<\beta<n-\alpha$,
\begin{equation*}
\left((-\Delta)^\frac{\alpha}{2}- \frac{\gamma}{|x|^\alpha} \right)|x|^{-\beta} =0\text{ in } \mathcal{S}'(\R^n)\hbox{ if and only if }\beta \in \{\Bp , \Bm\},
\end{equation*}
where $0<\Bm<\frac{n-\alpha}{2}$ is as in Proposition \ref{rem:M-1} and $\Bp:=n-\alpha-\Bm\in \left(\frac{n-\alpha}{2},n-\alpha\right)$. In particular, it follows from Proposition \ref{rem:M-1} that $\Bm,\Bp$ are the only solutions to $\Psi_{n,\alpha}(\beta) =\gamma$ in $(0,n-\alpha)$.
Since $0<\Bm<\frac{n-\alpha}{2}<\Bp<n-\alpha$, we get that $x\mapsto |x|^{-\Bm}$ is locally in $\h$. It is the``small" or variational solution, while $x\mapsto |x|^{-\Bp}$ is the``large" or singular solution. We extend $\Bm,\Bp$ to the whole interval $[0,\gamma_H(\alpha)]$ by defining 
\begin{equation}\label{def:b:0}
\beta_-(0):=0,\quad \beta_+(0):=n-\alpha, \quad  {\rm and}\quad \beta_-(\gamma_H(\alpha))=\beta_+(\gamma_H(\alpha))=\frac{n-\alpha}{2},
\end{equation}
which is consistant with Proposition \ref{rem:M-1}.

\smallskip\noindent We now proceed to define a critical threshold $\gamma_{crit}(\alpha)$ as follows. Assuming first that $n > 2 \alpha$, then $\frac{n-\alpha}{2}<\frac{n}{2}<n-\alpha$ and therefore, by Proposition \ref{rem:M-1}, there exists $\bar{\gamma}(\alpha) \in (0, \gamma_H(\alpha))$ such that 
\begin{equation*}
\left\{\begin{array}{rll}
\displaystyle \frac{n}{2} < \Bp < n-\alpha &\text{if}  \ \gamma \in (0,\bar{\gamma}(\alpha) )\vspace{0.1cm}\\
\displaystyle \Bp =  \frac{n}{2} \,\,\,\,\,\,\,\,\,\,\,&\text{if} \ \gamma = \bar{\gamma} (\alpha)\vspace{0.1cm}\\
 \displaystyle \frac{n- \alpha}{2} < \Bp < \frac{n}{2}\,\,\,\,\,\,\,\,\,\,\,\, &\text{if} \ \gamma \in (\bar{\gamma}(\alpha) , \gamma_H(\alpha)).
\end{array}\right.
\end{equation*}
 \noindent We then set 
\begin{equation}\label{def:gamma:crit}
\gamma_{crit}(\alpha):= \left\{\begin{array}{rll}
\displaystyle  \bar{\gamma} (\alpha) &\text{if } \ n > 2 \alpha \vspace{0.1cm}\\
 \displaystyle 0 \hfill  &\text{if } n =2\alpha\vspace{0.1cm}\\
\displaystyle -1 \hfill  &\text{if } n <2\alpha.
\end{array}\right.
\end{equation}
One can easily check that for $\gamma\in [0,\gamma_H(\alpha))$, we have that
\begin{equation*}
\gamma \in (\gamma_{crit}(\alpha), \gamma_H(\alpha)) \quad \Leftrightarrow \quad \Bp<\frac{n}{2}\quad \Leftrightarrow \; x\mapsto |x|^{-\Bp}\in L^2_{loc}(\rn).
\end{equation*}

\medskip\noindent We now introduce the following terminology in defining a notion of solution on a punctured domain. 

\begin{definition}\label{def:punct} Let $\Omega$ be a smooth domain (not necessarily bounded) of $\rn$, $n>1$. Let $f$ be a function in $ L^1_{loc}(\Omega\setminus\{0\})$. We say that $u: \Omega\to \rr$ is a solution to 
\begin{equation*}
\left\{\begin{array}{rl}
(-\Delta)^{\frac{\alpha}{2}}u=f & \text{in }  {\Omega \setminus \{0\}}\\
 u=0 & \text{in } \partial \Omega,
\end{array}\right.
\end{equation*}
provided
\begin{enumerate} 
\item For any $\eta\in C^\infty_c(\rn\setminus\{0\})$, we have that $\eta u\in \homega$;
\item $\int_{\Omega}\frac{|u(x)|}{1+|x|^{n+\alpha}}\, dx<\infty;$
\item For any $\varphi\in C^\infty_c(\Omega\setminus\{0\})$, we have that
\begin{equation*}
\frac{C_{n,\alpha}}{2}\int_{\mathbb{R}^n}\int_{\mathbb{R}^n} \frac{(u(x)-u(y))(\varphi(x)-\varphi(y))}{|x-y|^{n+\alpha}}\ dxdy=\int_{\mathbb{R}^n}f(x)\varphi(x)\, dx.
\end{equation*}
\end{enumerate}
\end{definition}
Note that the third condition is consistent thanks to the two preceding it. If $\Omega$ is bounded, the second hypothesis rewrites as $u\in L^1(\Omega)$.

\section{Profile of solutions}

Throughout this paper, we shall frequently use the following fact:
\begin{proposition}\label{prop:carac:h} A
measurable function  $u: \rn\to \rr$ belongs to $\h$ if and only if  $\int_{\rn}|u|^{2^\star_\alpha(0)}\, dx<+\infty$ and $\int_{(\rn)^2}\frac{|u(x)-u(y)|^2}{|x-y|^{n+\alpha}}\, dxdy<+\infty$. 
\end{proposition}
The proof consists of approximating $u$ by a compactly supported function satisfying the same properties. Then, by convoluting with a smooth mollifier,  this approximation is achieved by a smooth compactly supported function. The rest is classical and the details are left to the reader.

\medskip\noindent To prove Theorem \ref{th:asymp:ext}, we shall use a similar argument as in Dipierro-Montoro-Peral-Sciunzi \cite{DMPS}. The main idea is to transform  problem \eqref{Main:problem} into a different nonlocal problem in a weighted fractional space by using a representation introduced in Frank-Lieb-Seiringer \cite{Frank-Lieb-Seiringer}.

\begin{lemma}[Ground State Representation \cite{Frank-Lieb-Seiringer}; Formula (4.3)]\label{lem:ground:rep}
Assume $0<\alpha<2, 
$ $ n > \alpha$, $ 0<\beta< \frac{n-\alpha}{2}$. For $u\in C_c^\infty(\mathbb{R}^n \backslash \{0\})$, we let $v_\beta(x)=|x|^\beta u(x)$ in $\mathbb{R}^n\backslash\{0\}.$ Then,
\begin{eqnarray*}
\frac{C_{n,\alpha}}{2}\int_{\mathbb{R}^n}\int_{\mathbb{R}^n} \frac{|u(x)-u(y)|^2}{|x-y|^{n+\alpha}}\ dxdy&=&\Psi_{n,\alpha}(\beta)\int_{\mathbb{R}^n} \frac{u^2(x)}{|x|^\alpha}\ dx\\
&&+\frac{C_{n,\alpha}}{2}\int_{\mathbb{R}^n}\int_{\mathbb{R}^n} \frac{|v_\beta(x)-v_\beta(y)|^2}{|x-y|^{n+\alpha}}\ \frac{dx}{|x|^\beta}\frac{dy}{|y|^\beta}.
\end{eqnarray*}
\end{lemma}
\noindent Let now $u\in \h$ be a positive weak solution to \eqref{Main:problem}. Then by  (\ref{def:weak:sol}) and Remark 4.4 in \cite{DMPS},  we have
\[\frac{C_{n,\alpha}}{2}\int_{\mathbb{R}^n}\int_{\mathbb{R}^n} \frac{|u(x)-u(y)|^2}{|x-y|^{n+\alpha}}\ dxdy=\gamma\int_{\mathbb{R}^n}\frac{ u^2(x)}{|x|^\alpha}\ dx+\int_{\mathbb{R}^n}\frac{u^{2_\alpha^*(s)}(x)}{|x|^s}\ dx.\]
Set $v(x)=|x|^{\Bm}u(x)$ on $\mathbb{R}^n\backslash\{0\}$. It follows from Lemma \ref{lem:ground:rep} and the definition of $\Bm$ that 
\begin{eqnarray*}
\frac{C_{n,\alpha}}{2}\int_{\R^n}\int_{\R^n} \frac{|v(x)-v(y)|^2}{|x-y|^{n+\alpha}}\ \frac{dx}{|x|^{\Bm}}\frac{dy}{|y|^{\Bm}}
& =&\frac{C_{n,\alpha}}{2}\int_{\R^n}\int_{\R^n} \frac{|u(x)-u(y)|^2}{|x-y|^{n+\alpha}}\ dxdy\\
&&-\Psi_{n,\alpha}(\Bm)\int_{\R^n} \frac{u^2(x)}{|x|^\alpha}\ dx\\
&=&\gamma\int_{\R^n}\frac{u^2(x)}{|x|^\alpha}\ dx+\int_{\R^n}\frac{u^{2_\alpha^*(s)}(x)}{|x|^s}\ dx\\
&&-\Psi_{n,\alpha}(\Bm)\int_{\R^n} \frac{u^2(x)}{|x|^\alpha}\ dx\\
&=&\int_{\R^n}\frac{u^{2_\alpha^*(s)}(x)}{|x|^s}\,dx
=\int_{\R^n}\ \frac{v^{2_\alpha^*(s)}(x)}{|x|^{s+\Bm 2_\alpha^*(s)}} dx.
\end{eqnarray*}
For $ 0<\beta<\frac{n-\alpha}{2}$, define the space $\displaystyle H^{\frac{\alpha}{2},\beta}_0(\R^n)$ as the completion of $C_c^\infty(\R^n \setminus \{0\})$ with respect to the norm
\[\Vert \phi\Vert_{H^{\frac{\alpha}{2},\beta}_0(\R^n)}:=\left( \int_{\R^n}\int_{\R^n} \frac{|\phi(x)-\phi(y)|^2}{|x-y|^{n+\alpha}}\ \frac{dx}{|x|^{\beta}}\frac{dy}{|y|^{\beta}}\right)^{\frac{1}{2}}.\]
Many of the properties of the space $\displaystyle H^{\frac{\alpha}{2},\beta}_0(\R^n)$ were established in \cite{DV}. By  Lemma \ref{lem:ground:rep}, Remark 4.4 in \cite{DMPS} and \cite{AB}, we have that $\displaystyle v\in \displaystyle H^{\frac{\alpha}{2},\beta}_0(\R^n)$. Now, we introduce the operator $\displaystyle (-\Delta_{\beta})^{\frac{\alpha}{2}}$, whose action on a function $w$ is given via the following duality: For $\displaystyle \phi\in H^{\frac{\alpha}{2},\beta}_0(\R^n)$,
\[\langle (-\Delta_\beta)^{\frac{\alpha}{2}} w,\phi\rangle=\frac{C_{n,\alpha}}{2}\int_{\R^n}\int_{\R^n} \frac{(w(x)-w(y))(\phi(x)-\phi(y))}{|x-y|^{n+\alpha}}\ \frac{dx}{|x|^\beta}\frac{dy}{|y|^{\beta}}. 
\]
This means that $v$ is a weak solution to
\begin{eqnarray}\label{eqn:M-6}
(-\Delta_{\Bm})^{\frac{\alpha}{2}} v=\frac{v^{2_\alpha^*(s)-1}}{|x|^{{s+\Bm2_\alpha^*(s)}}} \quad\textnormal{in $\R^n$},
\end{eqnarray}
in the sense that for any $\phi\in H^{\frac{\alpha}{2},\Bm}(\R^n)$, we have that
\[\frac{C_{n,\alpha}}{2}\int_{\R^n}\int_{\R^n} \frac{(v(x)-v(y))(\phi(x)-\phi(y))}{|x-y|^{n+\alpha}}\ \frac{dx}{|x|^{\Bm}}\frac{dy}{|y|^{\Bm}}=\int_{\R^n} \frac{v^{2_\alpha^*(s)-1}}{|x|^{s+\Bm 2_\alpha^*(s)}}\phi\ dx .  \]
The following proposition gives a regularity result and a Harnack inequality for weak solutions of \eqref{eqn:M-6}. 
\begin{proposition}\label{prop:M-1}
Assume $0<s<\alpha<2,$ $n >\alpha$ and $0<\beta<\frac{n-\alpha}{2}$, and let $\displaystyle v\in H^{\frac{\alpha}{2},\beta}_0(\R^n)$ be a non-negative, non-zero weak solution to the problem
\[(-\Delta_\beta)^{\frac{\alpha}{2}} v=\frac{v^{2_\alpha^*(s)-1}}{|x|^{{s+\beta2_\alpha^*(s)}}} \quad\textnormal{in $\R^n$}.\]
\noindent Then, $v\in L^\infty(\R^n)$ and there exist constants $R>0$ and $C>0$ such that $\displaystyle C\leq v(x)$ in $B_R(0)$.

\end{proposition}

\begin{proof}
The statement that $v(x)\geq C$ in $B_R(0)$ is essentially the Harnack inequality for super-harmonic functions associated to the nonlocal operator $\displaystyle (-\Delta_\beta)^{\frac{\alpha}{2}}$, which is just Theorem 3.4 in Abdellaoui-Medina-Peral-Primo \cite{AMPP.CZ}. See also the proof of Lemma 3.10 in \cite{AMPP.CZ} and also \cite{DMPS}. We now show that $v\in L^\infty(\R^n)$ by using a similar argument as in  \cite{DMPS}. For any $p\geq 1$ and $T>0$, define the function
\[\phi_{p,T}(t)=\left\{
\begin{array}{ll}
\displaystyle t^{p} &\textnormal{if $0\leq t\leq T$} \vspace{0.1cm}\\
pT^{p-1}(t-T)+T^p  &\textnormal{if $t>T$}.
\end{array}
\right.\]
It is easy to check that the function $\phi_{p,T}(t)$ has the following properties:
\begin{itemize}
\item $\phi_{p,T}(t)$ is convex and Lipschitz in $[0,\infty)$.

\item $\phi_{p,T}(t)\leq t^p$ for all $t\geq0$.\vspace{0.1cm}

\item $t\phi_{p,T}'(t)\leq 2p\phi_{p,T}(t)$ for all $t\geq0$, since 
$\displaystyle t\phi_{p,T}'(t)=\left\{
\begin{array}{ll}
p\phi_{p,T}(t)  &\textnormal{if $0<t<T$}\vspace{0.1cm}\\
p T^{p -1} t  &\textnormal{if $t>T$}.
\end{array}
\right.$ 

\item If $T_2>T_1>0$, then $\phi_{p,T_1}(t)\leq\phi_{p,T_2}(t)$ for all $t\geq0$.

\end{itemize}
Since $\phi_{p,T}(t)$ is convex and Lipschitz, then as noted in \cite{lpps}, 
\begin{eqnarray}\label{eqn:M-7}
(-\Delta_\beta)^{\frac{\alpha}{2}} \phi_{p,T}(v)\leq \phi_{p,T}'(v)(-\Delta_\beta)^{\frac{\alpha}{2}}v\quad\textnormal{in $\R^n$}.
\end{eqnarray}
Since $\phi_{p,T}(t)$ is Lipschitz and $\phi_{p,T}(0) = 0$, then $\displaystyle \phi_{p,T}(v)\in H^{\frac{\alpha}{2},\beta}_0(\R^n)$. By the weighted fractional Hardy-Sobolev inequality, the ground state representation formula, Lemma \ref{lem:ground:rep}, and (\ref{def:mu}), we get that there exists some constant $C_0>0$ which only depends on $n$, $\alpha$, $s$ and $\beta$ such that
\begin{eqnarray}\label{eqn:M-8}
\left[\int_{\R^n} \frac{|\phi_{p,T}(v)|^{2_\alpha^*(s)}}{|x|^{s+\beta 2_\alpha^*(s)}}\ dx\right]^{\frac{2}{2_\alpha^*(s)}}\leq \frac{C_0}{2}\int_{\R^n}\int_{\R^n} \frac{|\phi_{p,T}(v(x))-\phi_{p,T}(v(y))|^2}{|x-y|^{n+\alpha}}\ \frac{dx}{|x|^\beta}\frac{dy}{|y|^\beta}.
\end{eqnarray}
Since $\phi_{p,T}(t)\geq0$ for all $t\geq0$, we get from \eqref{eqn:M-7} that
\begin{eqnarray*}
\int_{\R^n}\int_{\R^n} \frac{|\phi_{p,T}(v(x))-\phi_{p,T}(v(y))|^2}{|x-y|^{n+\alpha}}\ \frac{dx}{|x|^\beta}\frac{dy}{|y|^\beta}&=&\int_{\R^n}\phi_{p,T}(v)(-\Delta_\beta)^{\frac{\alpha}{2}}\phi_{p,T}(v)\ dx\\
&\leq& \int_{\R^n}\phi_{p,T}(v)\phi_{p,T}'(v)(-\Delta_\beta)^{\frac{\alpha}{2}}v\ dx\\
&=&\int_{\R^n} \phi_{p,T}(v)\phi_{p,T}'(v) \frac{v^{2_\alpha^*(s)-1}}{|x|^{s+\beta2_\alpha^*(s)}}\ dx\\&\leq&2p\int_{\R^n}|\phi_{p,T}(v)|^2 \frac{v^{2_\alpha^*(s)-2}}{|x|^{s+\beta2_\alpha^*(s)}}\ dx.
\end{eqnarray*}
Note that the last inequality holds, since $t\phi_{p,T}'(t)\leq2 p\phi(t)$ for all $t\geq0$.
By \eqref{eqn:M-8}, we have
\begin{eqnarray}\label{eqn:M-9}
\left[\int_{\R^n} \frac{|\phi_{p,T}(v)|^{2_\alpha^*(s)}}{|x|^{s+\beta 2_\alpha^*(s)}}\ dx\right]^{\frac{2}{2_\alpha^*(s)}}\leq pC_0\int_{\R^n}|\phi_{p,T}(v)|^2 \frac{v^{2_\alpha^*(s)-2}}{|x|^{s+\beta2_\alpha^*(s)}}\ dx.
\end{eqnarray}
Letting $\displaystyle p_1=\frac{2_\alpha^*(s)}{2}$, then 
\begin{eqnarray}\label{eqn:M-10}
\left[\int_{\R^n} \frac{|\phi_{p_1,T}(v)|^{2_\alpha^*(s)}}{|x|^{s+\beta 2_\alpha^*(s)}}\ dx\right]^{\frac{2}{2_\alpha^*(s)}}\leq p_1C_0\int_{\R^n}|\phi_{p_1,T}(v)|^2 \frac{v^{2_\alpha^*(s)-2}}{|x|^{s+\beta2_\alpha^*(s)}}\ dx.
\end{eqnarray}
For $m>0$, a simple computation and H\"{o}lder's inequality yield that
\begin{eqnarray*}
p_1C_0\int_{\R^n}|\phi_{p_1,T}(v)|^2 \frac{v^{2_\alpha^*(s)-2}}{|x|^{s+\beta2_\alpha^*(s)}} dx
 &=&p_1C_0\int_{v(x)\leq m}|\phi_{p_1,T}(v)|^2 \frac{v^{2_\alpha^*(s)-2}}{|x|^{s+\beta2_\alpha^*(s)}}\ dx\\
 &&+p_1C_0\int_{v(x)>m}|\phi_{p_1,T}(v)|^2 \frac{v^{2_\alpha^*(s)-2}}{|x|^{s+\beta2_\alpha^*(s)}} dx\\
 &\leq&  p_1 C_0  m^{2_\alpha^*(s)-2} \int_{v(x)\leq m} \frac{|\phi_{p_1,T}(v)|^2}{|x|^{s+\beta 2_\alpha^*(s)}} dx\\
 &&+p_1C_0 \int_{v(x)>m} \frac{|\phi_{p_1,T}(v)|^2}{|x|^{\frac{2(s+\beta2_\alpha^*(s))}{2_\alpha^*(s)}}} \frac{v^{2_\alpha^*(s)-2}}{|x|^{s+\beta2_\alpha^*(s)-\frac{2(s+\beta2_\alpha^*(s))}{2_\alpha^*(s)}}} dx\\
&\leq& p_1 C_0  m^{2_\alpha^*(s)-2} \int_{\R^n} \frac{|\phi_{p_1,T}(v)|^2}{|x|^{s+\beta 2_\alpha^*(s)}} dx\\
&&+p_1C_0\left[\int_{v(x)>m}\frac{|\phi_{p_1,T}(v)|^{2_\alpha^*(s)}}{|x|^{s+\beta 2_\alpha^*(s)}} dx\right]^{\frac{2}{2_\alpha^*(s)}}  \left[\int_{v(x)>m} \frac{v^{2_\alpha^*(s)}}{|x|^{s+\beta 2_\alpha^*(s)}}\ dx\right]^{\frac{\alpha-s}{n-s}}\\
& \leq& p_1 C_0  m^{2_\alpha^*(s)-2} \int_{\R^n} \frac{|\phi_{p_1,T}(v)|^2}{|x|^{s+\beta 2_\alpha^*(s)}}\ dx\\
&&+p_1C_0\left[\int_{\R^n}\frac{|\phi_{p_1,T}(v)|^{2_\alpha^*(s)}}{|x|^{s+\beta 2_\alpha^*(s)}}\ dx\right]^{\frac{2}{2_\alpha^*(s)}}  \left[\int_{v(x)>m} \frac{v^{2_\alpha^*(s)}}{|x|^{s+\beta 2_\alpha^*(s)}}\ dx\right]^{\frac{\alpha-s}{n-s}}.
\end{eqnarray*}
Recall that  $\displaystyle v\in H^{\frac{\alpha}{2},\beta}_0(\R^n)$, hence $\displaystyle \int_{\R^n} \frac{v^{2_\alpha^*(s)}}{|x|^{s+\beta 2_\alpha^*(s)}}\ dx<\infty$. Thus, we can take a large $M_0\gg1$ and fix it in such a way that
\[p_1C_0\left[\int_{v(x)>M_0} \frac{v^{2_\alpha^*(s)}}{|x|^{s+\beta 2_\alpha^*(s)}}\ dx\right]^{\frac{\alpha-s}{n-s}}\leq \frac{1}{2}.\]
Since $\phi_{p_1,T}(t)\leq t^{p_1}$ for all $t\geq0$, then by \eqref{eqn:M-10} and the fact that  $\displaystyle p_1=\frac{2_\alpha^*(s)}{2}$, we get
\begin{eqnarray}
\left[\int_{\R^n} \frac{|\phi_{p_1,T}(v)|^{2_\alpha^*(s)}}{|x|^{s+\beta 2_\alpha^*(s)}}\ dx\right]^{\frac{2}{2_\alpha^*(s)}}
&\leq& 2p_1 C_0  M_0^{2_\alpha^*(s)-2} \int_{\R^n} \frac{|\phi_{p_1,T}(v)|^2}{|x|^{s+\beta 2_\alpha^*(s)}}\ dx\nonumber\\
&\leq& 2p_1 C_0  M_0^{2_\alpha^*(s)-2} \int_{\R^n} \frac{|v|^{2p_1}}{|x|^{s+\beta 2_\alpha^*(s)}}\ dx\nonumber\\
&=&2p_1 C_0  M_0^{2_\alpha^*(s)-2} \int_{\R^n} \frac{v^{2_\alpha^*(s)}}{|x|^{s+\beta 2_\alpha^*(s)}}\ dx. \label{eqn:M-11}
\end{eqnarray}
Let $C_1= 2C_0M_0^{2_\alpha^*(s)-2}.$
By taking $T\rightarrow\infty$ in \eqref{eqn:M-11} and applying Fatou's lemma, we get that
\begin{eqnarray*}
\left[\int_{\R^n} \frac{v^{p_12_\alpha^*(s)}}{|x|^{s+\beta 2_\alpha*(s)}}\ dx\right]^{\frac{2}{2_\alpha*(s)}}\leq p_1C_1 \int_{\R^n} \frac{v^{2_\alpha^*(s)}}{|x|^{s+\beta 2_\alpha^*(s)}}\ dx<\infty.
\end{eqnarray*}
Define now recursively the sequence $\displaystyle \{p_k\}_{k=2}^\infty$ as follows:
\begin{eqnarray}\label{eqn:M-12}
2p_{k+1}+2_\alpha^*(s)-2=p_k2_\alpha^*(s) \quad\textnormal{for all $k\geq 1$}.
\end{eqnarray}
Using  \eqref{eqn:M-9} and \eqref{eqn:M-12}, we have
\begin{eqnarray}
\left[\int_{\R^n} \frac{|\phi_{p_{k+1},T}(v)|^{2_\alpha^*(s)}}{|x|^{s+\beta 2_\alpha^*(s)}} dx\right]^{\frac{2}{2_\alpha^*(s)}}
&\leq& p_{k+1}C_0\int_{\R^n}|\phi_{p_{k+1},T}(v)|^2 \frac{v^{2_\alpha^*(s)-2}}{|x|^{s+\beta2_\alpha^*(s)}} dx\nonumber\\
&\leq& p_{k+1}C_0\int_{\R^n} v^{2p_{k+1}} \frac{v^{2_\alpha^*(s)-2}}{|x|^{s+\beta2_\alpha^*(s)}} dx\nonumber\\
&=&C_0p_{k+1} \int_{\R^n} \frac{v^{p_k2_\alpha^*(s)}}{|x|^{s+\beta2_\alpha^*(s)}} dx.\quad \label{eqn:M-13}
\end{eqnarray}
We also have used the fact that $\phi_{p_{k+1},T}(t)\leq t^{p_{k+1}}$ for all $t\geq0.$
By taking $T\rightarrow\infty$ in \eqref{eqn:M-13} and applying Fatou's lemma, we get that
\[\left[\int_{\R^n} \frac{v^{p_{k+1}2_\alpha^*(s)}}{|x|^{s+\beta2_\alpha^*(s)}} dx\right]^{\frac{2}{2_\alpha^*(s)}}\leq C_0p_{k+1}\int_{\R^n} \frac{v^{p_k2_\alpha^*(s)}}{|x|^{s+\beta2_\alpha^*(s)}}  dx\quad\textnormal{for all $k\geq 1$}.\]
Hence, by \eqref{eqn:M-12}, we obtain that
\begin{eqnarray*}
\left[\int_{\R^n} \frac{v^{p_{k+1}2_\alpha^*(s)}}{|x|^{s+\beta2_\alpha^*(s)}}\ dx\right]^{\frac{1}{2_\alpha^*(s)(p_{k+1}-1)}}&\leq& (C_0p_{k+1})^{\frac{1}{2(p_{k+1}-1)}}\left[\int_{\R^n} \frac{v^{p_k2_\alpha^*(s)}}{|x|^{s+\beta2_\alpha^*(s)}}\ dx\right]^{\frac{1}{2(p_{k+1}-1)}}\\
&=&(C_0p_{k+1})^{\frac{1}{2(p_{k+1}-1)}}\left[\int_{\R^n} \frac{v^{p_k2_\alpha^*(s)}}{|x|^{s+\beta2_\alpha^*(s)}}\ dx\right]^{\frac{1}{2_\alpha^*(s)(p_k-1)}}\quad .
\end{eqnarray*}
For $k\geq 1$, set
$$I_k:=\left[\int_{\R^n} \frac{v^{p_k2_\alpha^*(s)}}{|x|^{s+\beta2_\alpha^*(s)}}\ dx\right]^{\frac{1}{2_\alpha^*(s)(p_k-1)}}\hbox{ and }D_k=(C_0p_{k+1})^{\frac{1}{2(p_{k+1}-1)}}.$$
We have $I_{k+1}\leq D_k I_k$ for all $k\geq 1$, and 
\begin{eqnarray*}
\ln I_{k+1}&\leq&\ln D_k+\ln I_k\leq\sum_{j=1}^k \ln D_j+\ln I_1\leq\sum_{j=1}^k \frac{\ln C_0+\ln p_{j+1}}{2(p_{j+1}-1)}+\ln I_1.
\end{eqnarray*}
It follows from \eqref{eqn:M-12} that  $\displaystyle p_{k+1}=p_1^k(p_1-1)+1$ for all $k\geq 0$. This coupled with the fact that $p_1>1$ yield  
\begin{eqnarray*}
\ln I_{k+1}&\leq&\sum_{j=1}^k \frac{\ln C_0}{2 p_1^j(p_1-1)}+\sum_{j=1}^k \frac{\ln [p_1^j(p_1-1)+1]}{2p_1^j(p_1-1)}+\ln I_1\\
&\leq&\sum_{j=1}^k \frac{\ln C_0}{2 p_1^j(p_1-1)}+\sum_{j=1}^k \frac{\ln p_1^{j+1}}{2p_1^j(p_1-1)}+\ln I_1<C_2<\infty.
\end{eqnarray*}
For any fix $R\geq 1$, we then have
\[\left[\int_{|x|\leq R} \frac{v^{p_k2_\alpha^*(s)}}{|x|^{s+\beta2_\alpha^*(s)}}\ dx\right]^{\frac{1}{2_\alpha^*(s)(p_k-1)}}\leq I_k\leq e^{C_2}=:C_3 \quad\textnormal{for all $k\geq 1$}.\]
Since $s+\beta2_\alpha^*(s)>0$, we then get
\begin{eqnarray*}
\left[\int_{|x|\leq R} v^{p_k2_\alpha^*(s)}\ dx\right]^{\frac{1}{2_\alpha^*(s)p_k}}\leq C_3 R^{\frac{s+\beta 2_\alpha^*(s)}{2_\alpha^*(s)p_k}} \quad\textnormal{for all $k\geq 1$}.
\end{eqnarray*}
Since $\displaystyle \lim_{k\rightarrow\infty}\ p_k=\infty$, we have

\[\Vert v\Vert_{L^\infty(B_R(0))}=\lim_{k\rightarrow\infty}\ \left[\int_{|x|\leq R} v^{p_k2_\alpha^*(s)}\ dx\right]^{\frac{1}{2_\alpha^*(s)p_k}}\leq C_3,\]
and finally, that $\displaystyle \Vert v\Vert_{L^\infty(\R^n)}\leq C_3$.\end{proof}

\begin{proof}[Proof of  Theorem \ref{th:asymp:ext}]
Let $v(x)=|x|^{\Bm}u(x)$ in $\R^n\backslash\{0\}$, by the discussion before at the beginning of section 3, we know that $\displaystyle v\in H^{\frac{\alpha}{2},\beta}_0(\R^n)$ is a positive weak solution to \eqref{eqn:M-6}. We deduce from Proposition \ref{prop:M-1} that for all $R>0,$ there exist some constant $C>1$ such that $C^{-1}\leq v(x)\leq C$ in $B_R(0)$. Since $v(x)=|x|^{\Bm}u(x)$ in $\R^n\backslash\{0\}$, then 
\begin{equation}\label{control:u:0}
\frac{C^{-1}}{|x|^{\Bm}}\leq u(x)\leq \frac{C}{|x|^{\Bm}} \text{ in }B_R(0)\backslash\{0\}.
\end{equation}
In order to prove the asymptotic behavior at zero, it is enough to show that $\lim\limits_{x \to 0} |x|^{\Bm} u(x) $ exists. To that end, we proceed as follows:

\medskip\noindent{\it Claim 1:} $u\in C^1(\rn\setminus\{0\})$.

\smallskip\noindent This is consequence of regularity theory and we only sketch the proof. First we define $f_0(x):=\gamma|x|^{-\alpha}u+u^{\crits-1}|x|^{-s}$, so that for any $\omega\subset\subset \rn\setminus\{0\}$, we have that $(-\Delta)^{\alpha/2}u=f_0$ in $\omega$ in the sense that $u\in \h$ and
$$\frac{C_{n,\alpha}}{2}\int_{\R^n} \int_{\R^n} \frac{|u(x) - u(y) |^2}{|x-y|^{n+\alpha}} dx dy =\int_\omega f_0\varphi\, dx\hbox{ for all }\varphi\in C^\infty_c(\omega).$$
It follows from \eqref{control:u:0} that $f_0\in L^\infty(\omega)$. Since $u\geq 0$ and $f_0\in L^\infty(\omega)$, it follows from Remark 2.5 (see also Theorem 2.1) in Jin-Li-Xiong \cite{Li_JEMS} that there exists $\tau>0$ such that $u\in C^{0,\tau}_{loc}(\rn\setminus\{0\})$. Then, using recursively Theorem 2.1 in Jin-Li-Xiong \cite{Li_JEMS}, we get that $u\in C^1(\rn\setminus\{0\})$. This proves the claim.

\medskip\noindent{\it Claim 2:} There exists $C>0$ such that $|x|^{\Bm+1}|\nabla u(x)|\leq C$ for all $x\in B_1(0)\setminus\{0\}$.

\smallskip\noindent If not, then there exists a sequence $(x_i)_{i\in\nn}\in B_1(0)\setminus\{0\}$ such that $\lim_{i\to +\infty}|x_i|^{\Bm+1}|\nabla u(x_i)|=+\infty$. For simplicity, we write $\bm:=\Bm$. It follows from from Claim 1,  that $\lim_{i\to +\infty}x_i=0$. We define $r_i:=|x_i|$ and we set
$$u_i(x):=r_i^{\bm}u(r_i x)\hbox{ for all }x\in \rn\setminus\{0\}.$$
It is easy to see that $u_i\in \h$, $u_i\geq 0$ for all $i\in\nn$ and  $(-\Delta)^{\alpha/2}u_i=f_i$ in $\omega\subset\subset \rn\setminus\{0\}$ where
$f_i(x):=\gamma|x|^{-\alpha}u_i+r_i^{(\crits-2)(\frac{n-\alpha}{2}-\bm)}u_i^{\crits-1}|x|^{-s}$ for all $x\in \rn\setminus\{0\}$. Using the apriori bound of Remark 2.5 (see also Theorem 2.1) in Jin-Li-Xiong \cite{Li_JEMS}, we get that there exists $\tau>0$ such that for any $R>1$, there exists $C(R)>0$ such that $\Vert u_i\Vert_{C^{0,\tau}(B_{R}(0)-B_{R^{-1}}(0))}\leq C(R)$ for all $i\in\nn$. Using recursively Theorem 2.1 of \cite{Li_JEMS} as in Step 1, we get that for any $\omega\subset\subset \rn\setminus\{0\}$, there exists $C(\omega)>0$ such that $\Vert u_i\Vert_{C^1(\omega)}\leq C(\omega)$. Taking $\omega$ large enough and estimating $|\nabla u_i(\frac{x_i}{|x_i|})|$, we get a contradiction, which proves Claim 2.

Set now $h(x):=\frac{u^{\crits-2}}{|x|^s}$, so that
$(-\Delta)^{\alpha/2}u-\frac{\gamma}{|x|^\alpha}u=h(x)u\hbox{ in }\rn.$
It follows from Claims 1 and 2, that $h\in C^1(\rn\setminus\{0\})$, and for some $C>0$,  
$$|h(x)|+|x|\cdot|\nabla h(x)|\leq C|x|^{\theta-\alpha}\hbox{ for all }x\in B_1(0)\setminus\{0\},$$
where $\theta:=(\crits-2)(\frac{n-\alpha}{2}-\bm)>0$. It then follows from Lemma \ref{lem:ext:fall:felli} below that there exists $\lambda_0>0$ such that 
$$\lim_{x\to 0}|x|^{\bm}u(x)=\lambda_0>0.$$
In order to deal with the behavior at infinity, let $w$ be the fractional Kelvin transform of $u$, that is, 
$$ w(x)=|x|^{\alpha-n}u(x^*):=|x|^{\alpha-n}u\left(\frac{x}{|x|^2}\right) \hbox{in $\R^n\backslash\{0\}.$}$$
 By  Lemma 2.2 and  Corollary 2.3 in \cite{Fall-Weth}, we have that $\displaystyle w\in H^{\frac{\alpha}{2}}(\R^n)$. A simple calculation gives us that $w$ is also a positive weak solution to \eqref{Main:problem}. Indeed, we have
\begin{align*}
(-\Delta)^{\frac{\alpha}{2} } w(x) = \frac{1}{|x|^{n+\alpha}} \left((-\Delta)^{\frac{\alpha}{2} }u\right)\left(\frac{x}{|x|^2}\right)
= \gamma \frac{w(x)}{|x|^\alpha}  + \frac{w^{2_\alpha^*(s)-1}(x)}{|x|^s}.
\end{align*}
Arguing as in the first part of the proof, we get that there exists $\lambda_\infty>0$ such that
$$\lim_{x\to 0}|x|^{\Bm}w(x)=\lambda_\infty>0.$$
Coming back to $u$, this implies that 
$$\lim_{|x|\to \infty}|x|^{\Bp}u(x)=\lambda_\infty>0.$$
This ends the proof of Theorem \ref{th:asymp:ext}.
\end{proof}
\section{ Analytic Conditions for The  Existence of Extremals} 

Let $a\in C^{0,\tau}(\overline{\Omega})$ for some $\tau\in(0,1)$, and define the functional $J_{a}^{\Omega}: \homega \longrightarrow \R$ by
$$J_{a}^{\Omega}(u):=\frac{\frac{C_{n,\alpha}}{2}\int_{\R^n} \int_{\R^n} \frac{|u(x) - u(y) |^2}{|x-y|^{n+\alpha}} dx dy -\gamma \int_{\Omega} \frac{u^2}{|x|^{\alpha}}dx -  \int_{\Omega} a u^2 dx }{\left(\int_{\Omega} \frac{|u|^{\crits}}{|x|^s}dx\right)^{\frac{2}{\crits}}},$$
in such a way that 
$$\mu_{\gamma,s,\alpha,a}(\Omega) := \inf \left\{ J_{a}^{\Omega}(u) : u \in \homega \setminus \{0\}  \right\}.
$$

We now prove the following proposition, which gives analytic conditions for the  existence of extremals  for $\mu_{\gamma,s,\alpha,a} (\Omega).$
 
\begin{proposition}\label{prop:exist}
Let $\Omega$ be a bounded domain in $\R^n$ $ (n > \alpha)$ such that $0 \in \Omega,$ and assume that $0 \le \gamma < \gamma_H(\alpha)$ and $0\le s \le \alpha.$
\begin{enumerate}
\item  If $\mu_{\gamma,s,\alpha,a} (\Omega) <\mu_{\gamma,s,\alpha}(\R^n) $, then there are extremals for $\mu_{\gamma,s,\alpha,a}(\Omega)$ in $\homega.$
 
\item If $a(x)$ is a constant $\lambda$, with $0<\lambda < \lambda_1(\LO)$ and if $s < \alpha,$ then $\mu_{\gamma,s,\alpha,a}(\Omega)>0.$
\end{enumerate}
\end{proposition}

\begin{proof}
Let $(u_k)_{k \in \N} \subset \homega \setminus \{0\}$ be a minimizing sequence for $ \mu_{\gamma,s,\alpha,a} (\Omega),$ that is,
$$ J^\Omega_a (u_k) = \mu_{\gamma,s,\alpha,a}(\Omega) +o(1) \ \text{ as } k \to \infty. $$ 
Up to multiplying by a constant, we may assume that 
\begin{equation}\label{minim:norm}
\int_{\Omega} \frac{|u_k|^{\crits}}{|x|^{s}}\ dx=1
\end{equation}
\begin{equation} \label{lim:minim}
\frac{C_{n,\alpha}}{2}\int_{\R^n} \int_{\R^n} \frac{|u_k(x) - u_k(y) |^2}{|x-y|^{n+\alpha}} dx dy - \int_\Omega  \left(  \frac{\gamma}{|x|^\alpha}  + a  \right) u^2_k dx  = \mu_{\gamma,s,\alpha,\lambda}(\Omega) +o(1) 
\end{equation}
as $k\to +\infty$. By (\ref{minim:norm}), we have $\displaystyle \int_\Omega
 u_k^2 dx \le C <\infty \text{ for all } k.$ Since $0 \le \gamma<\gamma_H(\alpha),$ the fractional Hardy inequality combined with (\ref{lim:minim}) yields that 
$\displaystyle \|u_k\|_ {\homega} \le C$ for all $k$. It then follows that there exists $u \in \homega$ such that, up to a subsequence, such that $(u_k)$ goes to $u$ weakly in $\homega$ and strongly in $L^2(\Omega)$ as $k \to \infty.$

\smallskip\noindent We first show that $ \int_{\Omega} \frac{|u|^{\crits}}{|x|^{s}}dx =1.$ Define $\theta_k = u_k -u $ for all $k \in \N$. It follows from the boundedness in $\homega$ that, up to a subsequence, we have that $\theta_k\rightharpoonup 0$ weakly in $\homega$, strongly in $L^2(\Omega)$ as $k \to \infty$, and $\theta_k(x)\to 0$ for a.e. $x\in\Omega$ as $k\to +\infty$. Hence, by the Brezis-Lieb lemma (see \cite{Brezis-Lieb} and \cite{Yang}), we get that 
\begin{equation*}
\int_{\R^n} \int_{\R^n} \frac{|u_k(x) - u_k(y) |^2}{|x-y|^{n+\alpha}} dx dy = \int_{\R^n} \int_{\R^n} \frac{|\theta_k(x) - \theta_k(y) |^2}{|x-y|^{n+\alpha}} dx dy +\int_{\R^n} \int_{\R^n} \frac{|u(x) - u(y) |^2}{|x-y|^{n+\alpha}} dx dy +o(1),
\end{equation*}
\begin{equation*}
  1=\int_{\Omega} \frac{|u_k|^{\crits}}{|x|^{s}}dx= \int_{\Omega} \frac{|\theta_k|^{\crits}}{|x|^{s}}dx+ \int_{\Omega} \frac{|u|^{\crits}}{|x|^{s}}dx+o(1),
\end{equation*}
$$\int_{\Omega} \frac{u_k^{2}}{|x|^{\alpha}}dx= \int_{\Omega} \frac{\theta_k^{2}}{|x|^{\alpha}}dx+ \int_{\Omega} \frac{u^{2}}{|x|^{\alpha}}dx+o(1),\hbox{ and } \int_\Omega u_k^2 dx = \int_\Omega u^2 dx +o(1), $$
as $k \to \infty.$ Thus, we have
\begin{align} \label{eq:BL:lemma}
\begin{split}
\mu_{\gamma,s,\alpha, a}(\Omega)&=\left[ \frac{C_{n,\alpha}}{2}\int_{\R^n} \int_{\R^n} \frac{|u(x) - u(y) |^2}{|x-y|^{n+\alpha}} dx dy - \int_\Omega  \left(  \frac{\gamma}{|x|^\alpha}  + a  \right) u^2 dx \right] \\
&+ \left[ \frac{C_{n,\alpha}}{2}\int_{\R^n} \int_{\R^n} \frac{|\theta_k(x) - \theta_k(y) |^2}{|x-y|^{n+\alpha}} dx dy - \gamma \int_\Omega  \frac{\theta_k^2}{|x|^\alpha} dx \right]+  o(1)
\end{split}
\end{align}
as $k\to +\infty$. The definition of $\mu_{\gamma,s,\alpha,a}(\Omega)$ and $\homega\subset \h$ yield
\begin{equation*}
 \mu_{\gamma,s,\alpha,a}(\Omega) \left( \int_{\Omega}\frac{|u|^{\crits}}{|x|^{s}}dx   \right)^{\frac{2}{2_\alpha^*(s)}} \le \frac{C_{n,\alpha}}{2}\int_{\R^n} \int_{\R^n} \frac{|u(x) - u(y) |^2}{|x-y|^{n+\alpha}} dx dy - \int_\Omega  \left(  \frac{\gamma}{|x|^\alpha}  +a  \right) u^2 dx,
\end{equation*}
and
\begin{equation}\label{HS:ineq:theta}
 \mu_{\gamma,s,\alpha}(\R^n) \left( \int_{\Omega} \frac{|\theta_k|^{\crits}}{|x|^{s}}dx   \right)^{\frac{2}{2_\alpha^*(s)}} \le \frac{C_{n,\alpha}}{2}\int_{\R^n} \int_{\R^n} \frac{|\theta_k(x) - \theta_k(y) |^2}{|x-y|^{n+\alpha}} dx dy - \gamma \int_\Omega  \frac{\theta_k^2}{|x|^\alpha} dx.
\end{equation}
Summing these two inequalities and using (\ref{minim:norm}) and (\ref{eq:BL:lemma}), and passing to the limit $k \to \infty$, we obtain
$$ \mu_{\gamma,s,\alpha}(\R^n) \left(1- \int_{\Omega} \frac{|u|^{\crits}}{|x|^{s}}dx  \right)^{\frac{2}{2_\alpha^*(s)}} \le     \mu_{\gamma,s,\alpha,a}(\Omega) \left( 1 -  \left( \int_{\Omega}\frac{|u|^{\crits}}{|x|^{s}}dx   \right)^{\frac{2}{2_\alpha^*(s)}}\right). $$
Finally, the fact that  $  \mu_{\gamma,s,\alpha,a}(\Omega)< \mu_{\gamma,s,\alpha}(\R^n)$ implies that $ \int_{\Omega}\frac{|u|^{\crits}}{|x|^{s}}dx=1.$ It remains to show that $u$ is an extremal for $\mu_{\gamma,s,\alpha,a}(\Omega)$. For that, note that since $ \int_{\Omega}\frac{|u|^{\crits}}{|x|^{s}}dx=1,$ the definition of $\mu_{\gamma,s,\alpha,a}(\Omega)$ yields that
$$\mu_{\gamma,s,\alpha,a}(\Omega) \le \frac{C_{n,\alpha}}{2}\int_{\R^n} \int_{\R^n} \frac{|u(x) - u(y) |^2}{|x-y|^{n+\alpha}} dx dy - \int_\Omega  \left( \frac{\gamma }{|x|^\alpha}  + a \right) u^2 dx. $$
The second term in the right-hand-side of (\ref{eq:BL:lemma}) is nonnegative due to (\ref{HS:ineq:theta}). Therefore, we get that 

$$ \mu_{\gamma,s,\alpha,a}(\Omega) = \frac{C_{n,\alpha}}{2}\int_{\R^n} \int_{\R^n} \frac{|u(x) - u(y) |^2}{|x-y|^{n+\alpha}} dx dy - \int_\Omega  \left( \frac{\gamma }{|x|^\alpha}  + a  \right)  u^2dx.$$ 
This proves the first claim of the Proposition.

\smallskip\noindent Now assume that $ \lambda\in (0,\lambda_1(\LO))$ and $ 0 \le \gamma < 
\gamma_H(\alpha)$, then for all $ u \in \homega \setminus \{0\},$ 

\begin{align*}
J_\lambda^{\Omega}(u) &= \frac{\frac{C_{n,\alpha}}{2}\int_{\R^n} \int_{\R^n} \frac{|u(x) - u(y) |^2}{|x-y|^{n+\alpha}} dx dy - \int_\Omega  \left(  \frac{\gamma}{|x|^\alpha}  + \lambda \right)u^2 dx}{ \left( \int_\Omega \frac{|u|^{\crits}}{|x|^{s}}dx   \right)^{\frac{2}{2_\alpha^*(s)}}}\\
&\ge \left( 1 - \frac{\lambda}{\lambda_1(\LO)} \right)  \frac{ \frac{C_{n,\alpha}}{2}\int_{\R^n} \int_{\R^n} \frac{|u(x) - u(y) |^2}{|x-y|^{n+\alpha}} dx dy - \gamma \int_\Omega   \frac{u^2}{|x|^\alpha} dx}{ \left( \int_\Omega \frac{|u|^{\crits}}{|x|^{s}}dx   \right)^{\frac{2}{2_\alpha^*(s)}}}\\
& \ge \left( 1 - \frac{\lambda}{\lambda_1(\LO)} \right) \left( 1- \frac{\gamma}{\gamma_H(\alpha)}\right) \mu_{0,s,\alpha,0}(\Omega) \\
& = \left( 1 - \frac{\lambda}{\lambda_1(\LO)} \right) \left( 1- \frac{\gamma}{\gamma_H(\alpha)}\right) \mu_{0,s,\alpha,0}(\R^n ) >0.
\end{align*}
Therefore, $\mu_{\gamma,s,\alpha,\lambda}(\Omega) > 0.$
\end{proof}

\section{The fractional  Hardy singular interior mass of a domain in the critical case }

In this section, we define the fractional  Hardy singular interior mass of a domain by proving Theorem \ref{Theorem:TheMass}. We shall need the following five lemmae.

\begin{lemma}\label{lem:dist} Assume $0<\beta\leq n-\alpha$, and let $\eta \in C^{\infty}_c(\Omega)$ be a cut-off function such that $0\le \eta(x) \le 1$ in $\Omega,$ and  $\eta(x) \equiv 1$ in $B_\delta(0),$ for some $\delta >0 $ small. Then $x\mapsto \eta(x)|x|^{-\beta}\in \homega$ and there exists $f_\beta\in L^\infty_{loc}(\rn)$ with $f_\beta(x)\geq 0$ on $B_\delta(0)$ and $f_\beta\in C^1(B_\delta(0))$ such that
\begin{equation}\label{eq:dist}
(- \Delta)^{\frac{\alpha}{2}} (\eta |x|^{-\beta})=\Phi_{n,\alpha}(\beta)|x|^{-\alpha}\eta|x|^{-\beta}+f_\beta \quad \hbox{ in }{\mathcal D}^\prime(\Omega\setminus\{0\}),
\end{equation}
in the sense that, if $v_\beta(x):=\eta(x)|x|^{-\beta}$, then for all $\varphi\in C^\infty_c(\Omega\setminus\{0\})$, 
$$ \frac{C_{n,\alpha}}{2}\int_{\R^n} \int_{\R^n} \frac{(v_\beta(x) - v_\beta(y)) (\varphi(x) -\varphi(y))}{|x-y|^{n+\alpha}} dx dy = \Phi_{n,\alpha}(\beta)\int_{\Omega} \frac{v_\beta\varphi}{|x|^{\alpha}}dx+\int_\Omega f_\beta\varphi(x)\, dx. $$
Moreover, if $\beta<\frac{n-\alpha}{2}$, then $v_\beta\in \homega$ and equality \eqref{eq:dist} holds in the classical sense of $\homega$.
\end{lemma}
\begin{proof} When $\beta<\frac{n-\alpha}{2}$, it follows from Proposition \ref{prop:carac:h} that $x\mapsto \eta(x)|x|^{-\theta}\in \homega$. In the general case, for $\varphi\in C^\infty_c(\Omega\setminus\{0\})$, straightforward computations yield
$$ \frac{C_{n,\alpha}}{2}\int_{\R^n} \int_{\R^n} \frac{(v_\beta(x) - v_\beta(y)) (\varphi(x) -\varphi(y))}{|x-y|^{n+\alpha}} dx dy = \langle (-\Delta)^{\frac{\alpha}{2}}|x|^{-\beta},\eta \varphi\rangle+\int_\Omega f_\beta\varphi\, dx,$$
where
$$f_\beta(x):=C(n,\alpha)\lim_{\eps\to 0}\int_{|x-y|>\eps}\frac{\eta(x)-\eta(y)}{|x-y|^{n+\alpha}}\cdot\frac{1}{|y|^{\beta}}\, dy \quad \hbox{for all $x\in \rn$.}$$
Note that $f_\beta\in L^\infty_{loc}(\rn)$, and for $x\in B_\delta(0)$, we have that
$$f_\beta(x):=C(n,\alpha)\int_{\rn}\frac{1-\eta(y)}{|x-y|^{n+\alpha}}\cdot\frac{1}{|y|^{\beta}}\, dy \geq 0,$$
yielding that $f_\beta\in C^1(B_{\delta}(0))$. Since $\varphi\equiv 0$ around $0$, the lemma is a consequence of \eqref{lap:x:theta:dist}.\end{proof}

\begin{lemma}[A comparison principle via coercivity]\label{lem:comp}
Suppose $\Omega$ be a bounded smooth domain in $\mathbb{R}^n,$  $0<\alpha<2,$ $\gamma < \gamma_H(\alpha)$  and $a(x)\in C^{0,\tau}(\overline{\Omega})$ for some $\tau\in(0,1)$. Assume that  the operator $(-\Delta)^{\frac{\alpha}{2}}-(\frac{\gamma}{|x|^\alpha}+a(x))$ is coercive.  Let  $u$ be a function in $\homega $ that satisfies 
\begin{eqnarray*}
\left\{
\begin{array}{ll}
\displaystyle 
(-\Delta)^{\frac{\alpha}{2}}u-\left(\frac{\gamma}{|x|^\alpha}+a(x)\right)u &\ge 0\quad \textnormal{in $\Omega$}\\
\hfill  u &\ge 0\quad \textnormal{on $\partial\Omega$},
\end{array}
\right.
\end{eqnarray*}
in the sense that $u\geq 0$ in $\rn\setminus\Omega$ and 
$$ \frac{C_{n,\alpha}}{2}\int_{\R^n} \int_{\R^n} \frac{(u(x) - u(y)) (v(x) -v(y))}{|x-y|^{n+\alpha}} dx dy - \gamma \int_{\Omega} \frac{u . v}{|x|^{\alpha}}dx - \int_{\Omega} a(x) u v dx\geq 0 $$
for all $v \in \homega$ with $v \ge 0 $  a.e. in $\Omega.$ Then, $u \ge 0 $ in $\Omega.$
\end{lemma}

\begin{proof}
Let $u_-(x) = - \min( u(x) , 0 )$ be the negative part of $u$.  It follows from Proposition \ref{prop:carac:h} that $u_-\in\homega$. We can therefore use it as a test function  to get 
$$   \langle L u, u_-\rangle:= \frac{C_{n,\alpha}}{2} \int_{\R^n} \int_{\R^n} \frac{(u(x) - u(y)) (u_-(x) -u_-(y))}{|x-y|^{n+\alpha}} dx dy - \gamma \int_{\Omega} \frac{u  u_-}{|x|^{\alpha}}dx - \int_{\Omega} a(x) u u_- dx\geq 0$$
Let 
$$ \Omega^+ :=\{ x : u(x) \ge 0 \} \quad \text{ and } \quad  \Omega^- := \{ x : u(x) < 0 \}. $$
Straightforward computations yield 
\begin{align*}
0\leq   - \langle L u_- , u_- \rangle & - \frac{C_{n,\alpha}}{2} \int_{\Omega^-} \int_{\Omega^+} \frac{(u(x) - u(y)) u_-(y)}{|x-y|^{n+\alpha}} dx dy \\
& + \frac{C_{n,\alpha}}{2} \int_{\Omega^+} \int_{\Omega^-} \frac{(u(x) - u(y)) u_-(x)}{|x-y|^{n+\alpha}} dx dy,
\end{align*}
which yields via coercivity
$$ c \|u_-\|^2_{\h} \le   \langle L u_- , u_- \rangle \le 0.  $$ 
Thus, $u_- \equiv 0$, and therefore, $u \ge 0 $ on $\Omega.$
\end{proof}

\begin{lemma}\label{lem:control:u}
Assume that $u \in \homega $ is a weak solution of 
\begin{equation*}
 (-\Delta)^{\frac{\alpha}{2}}u-\left(\frac{\gamma+O(|x|^\tau)}{|x|^\alpha}\right) u=0\quad \textnormal{ in }\homega,
\end{equation*}
for some $\tau > 0.$ If $u \not \equiv 0$ and $u \ge 0$, then there exists a constant $C > 0$ such that  
$$C^{-1} \le |x|^{\Bm} u(x) \le C  \  \text{ for } x \to 0, \ x \in \Omega. $$
\end{lemma}

\begin{proof}
We  use the weak Harnack inequality to  prove the lower bound. Indeed, using  Theorem 3.4 and Lemma 3.10 in \cite{AMPP.CZ}, we get that there exists $ C_1 > 0 $ such that for $ \delta_1 >0 $ small enough, 
\begin{equation*}
u(x) \ge C_1 |x|^{-\Bm} \text{ in } B_{\delta_1}.
\end{equation*}
The other inequality goes as in the iterative scheme used to prove Proposition \ref{prop:M-1}.\end{proof}

\begin{lemma}[See Fall-Felli \cite{Fall_Felli}]\label{lem:ext:fall:felli}
Consider an open subset $\omega\subset \Omega$ with $0\in \omega$, and a function $h\in C^1(\omega)$  such that for some $\tau>0$, 
$$|h(x)|+|x|\cdot|\nabla h(x)|\leq C|x|^{\tau-\alpha}\hbox{ for all }x\in\omega\setminus\{0\}.$$
Let $u \in \homega$ be a weak solution of \begin{equation*}
 (-\Delta)^{\frac{\alpha}{2}}u-\frac{\gamma}{|x|^\alpha} u=h(x)u\quad \textnormal{ in }\omega\subset\Omega,
\end{equation*}
in the sense that for all $\varphi\in C^\infty_c(\omega)$, 
$$\frac{C_{n,\alpha}}{2}\int_{\R^n} \int_{\R^n} \frac{(u(x) - u(y)) (\varphi(x)-\varphi(y))}{|x-y|^{n+\alpha}} dx dy - \gamma \int_{\Omega} \frac{u \varphi}{|x|^{\alpha}}dx = \int_{\Omega} h(x) u \varphi\, dx.$$
Assume further that there exists $C>0$ such that 
\begin{equation*}
  C^{-1}\le |x|^{\Bm} u(x) \le C \ \text{ for } x \to 0, \ x \in \Omega.
\end{equation*}
Then there exists  $l >0$ such that  
\begin{equation*}
\lim\limits_{x \to 0} |x|^{\Bm} u(x) = l.
\end{equation*}
\end{lemma}
\begin{proof} This result is an extension of Theorem 1.1 proved by Fall-Felli \cite{Fall_Felli}, who showed that under these conditions, one has
\begin{equation}\label{cv:u}
\lim_{\tau\to 0}\left|\tau x\right|^{\frac{n-\alpha}{2}-\sqrt{\left(\frac{n-\alpha}{2}\right)^2+\mu}}u(\tau x)=\psi\left(0,\frac{x}{|x|}\right)\hbox{ in }C^{1}_{loc}(B_1(0))\setminus\{0\}
\end{equation}
where $\mu\in\rr$ and $\psi:\mathbb{S}^{n+1}_+:=\{\theta\in \mathbb{S}^{n+1}\; \theta_1>0\}\to\rr$ are respectively an eigenvalue and an eigenfunctions for the problem
\begin{equation}\label{eq:eigen}\left\{\begin{array}{ll}
-\hbox{div}(\theta_1^{1-\alpha}\nabla \psi)=\mu\psi&\hbox{ in }\mathbb{S}^{n+1}_+\\
-\lim_{\theta_1\to 0}{\theta_1}^{1-\alpha}\partial_{\nu}\psi(\theta_1,\theta')=\gamma k_{\alpha/2}&\hbox{ for }\theta'\in\partial\mathbb{S}^{n+1}_+,
\end{array}\right.\end{equation}
where $k_{\alpha/2}$ is a positive constant. We refer to \cite{Fall_Felli} for the explicit definition of this eigenvalue problem, in particular the relevant spaces used via the Caffarelli-Silvestre classical representation \cite{Caffarelli-Silvestre}. It then follows from the pointwise control \eqref{control:u:0} that 
$$\Bm:=\frac{n-\alpha}{2}-\sqrt{\left(\frac{n-\alpha}{2}\right)^2+\mu},$$
and by Proposition 2.3 in Fall-Felli \cite{Fall_Felli}, that $\mu$ is the first eigenvalue of the eigenvalue problem \eqref{eq:eigen}. Then, using classical arguments, we get that the corresponding eigenspace is one-dimensional and is spanned by any positive eigenfunction of \eqref{eq:eigen} (no matter the value of $\mu$, it must necessarily be the first eigenvalue).

\smallskip\noindent We are  left with proving that $\psi(0,x/|x|)$ is independant of $x$. In view of the remarks above, this amonts to prove the existence of a positive eigenfunction that is constant on the boundary.

\medskip\noindent  We now exhibit such an eigenfunction by following the argument in Proposition 2.3 in \cite{Fall_Felli}. First, use (\cite{Fall}, Lemma 3.1) to obtain $\Gamma\in C^0([0,+\infty)\times\rn)\cap C^2((0,+\infty)\times \rn)$ such that
\begin{equation}\label{eq:eigen:2}\left\{\begin{array}{ll}
-\hbox{div}(t^{1-\alpha}\nabla \Gamma)=0&\hbox{ in }(0,+\infty)\times \rn\\
-\lim_{t\to 0}t^{1-\alpha}\partial_{\nu}\Gamma(t,x)=k_{\alpha/2}\frac{\gamma }{|x|^\alpha}&\hbox{ for }x\in \rn=\partial((0,+\infty)\times \rn)\\
\Gamma(0,x)=|x|^{-\Bm}&\hbox{ for }x\in \rn=\partial((0,+\infty)\times \rn).
\end{array}\right.\end{equation}
Moreover, $\Gamma$ is in the relevant function space, $\Gamma>0$ and satisfies
$$\Gamma(z)=|z|^{-\Bm}\Gamma\left(\frac{z}{|z|}\right)\hbox{ for all }z\in (0,+\infty)\times \rn$$
where $|z|=\sqrt{t^2+|x|^2}$ if $z=(t,x)$. In particular, we have that $\Gamma(z)=|z|^{-\Bm}\psi_0(\theta)$ for $\theta:=z/|z|$ and some $\psi_0\in C^0(\overline{\mathbb{S}^{n+1}_+})\cap C^2(\mathbb{S}^{n+1}_+)$. Following \cite{Fall_Felli}, we get that $\psi_0$ is an eigenvalue for the problem \eqref{eq:eigen}. Moreover, $\psi_0>0$. Therefore, $\psi_0$ corresponds to the first eigenvalue and spans the corresponding eigenspace. Finally, we remark that for $\theta\in \partial\mathbb{S}_+^{n+1}$, we have that
$$\psi_0(0,\theta)=\Gamma(0,\theta)=|\theta|^{-\Bm}=1.$$
Since the eigenspace is one-dimensional, there exists $l\in\rr$ such that $\psi=l\cdot \psi_0$. Therefore $\psi(0,x/|x|)=l$ for all $x\in B_1(0)\setminus\{0\}\subset\rn$. It then follows from \eqref{cv:u} that
$$\lim_{x\to 0}|x|^{\Bm}u(x)=l>0,$$
which complete the proof of Lemma \ref{lem:ext:fall:felli}.\end{proof}

\begin{proof}[Proof of Theorem \ref{Theorem:TheMass}]  We first prove the existence of a solution. For $\delta >0 $ small enough, let $\eta \in C^{\infty}_c(\Omega)$ be a cut-off function as in Lemma \ref{lem:dist} such that $\eta(x) \equiv 1$ in $B_\delta(0).$  Set $\beta:=\Bp\leq n-\alpha$ in \eqref{eq:dist} and define
$$f(x) := - \left( (- \Delta)^{\frac{\alpha}{2}} - \left(\frac{\gamma}{|x|^{\alpha}}+ a(x) \right) \right) (\eta |x|^{-\Bp})=-f_{\Bp}+\frac{a\eta}{|x|^{\Bp}} \text{ in } \Omega \setminus \{0 \}$$
in the distribution sense. In particular, $f\in C^1(B_{\delta}(0)\setminus \{0\})$ and there exists a positive constant $C>0$ such that  
\begin{equation}\label{bnd:f:0}
|f(x)|+|x|\cdot |\nabla f(x)| \le C |x|^{-\Bp} \ \text{ for } x \neq 0 \text{ close to } 0. 
\end{equation}
In the sequel, we write $\bp:=\Bp$ and $\bm:=\Bm$. Note that the assumption $\gamma > \gamma_{crit}(\alpha)$ implies that  $\bp < \frac{n}{2} < \frac{n+\alpha}{2}.$  Thus, using (\ref{bnd:f:0}) and the fact that  $\bp < \frac{n+\alpha}{2},$ we get that $f \in L^{\frac{2n}{n+\alpha}}(\Omega)$. Since  $L^{\frac{2n}{n+\alpha}}(\Omega) = \left( L^{\frac{2n}{n-\alpha}}(\Omega) \right)'  \subset \left( \homega \right)', $ there exists $g \in \homega$ such that 
$$  \left( (- \Delta)^{\frac{\alpha}{2}} - \left(\frac{\gamma}{|x|^{\alpha}}+ a(x) \right) \right)  g = f \text{ weakly  in } \homega .$$
\noindent Set
\begin{equation}\label{def:u:mass}
H(x) := \frac{\eta(x)}{|x|^{\bp}} + g(x) \text{ for all }  x \in \overline{\Omega} \setminus \{0\}.
\end{equation}
Thanks to \eqref{eq:dist}, $H:\Omega\to\rr$ is a solution to
\begin{equation}\label{eq:u:mass}
\left\{\begin{array}{rl}
(-\Delta)^{\frac{\alpha}{2}}H-\left( \frac{\gamma}{|x|^\alpha}+a(x)\right) H=0 & \text{in }  {\Omega \setminus \{0\}}\\
 H=0 & \text{in } \R^n \setminus \Omega,
\end{array}\right.
\end{equation}
in the sense of Definition \ref{def:punct}. The idea is to now write $f$ as the difference of two positive $C^1$ functions. The decomposition $f=|f|-2f_-$ does not work here since the resulting functions are not necessarily $C^1$. To  smooth out the functions $x\mapsto |x|$ and $x\mapsto x_-$, we consider
$$\varphi_1(x):=\sqrt{1+x^2}\hbox{ and }\varphi_2(x):=\varphi_1(x)-x\hbox{ for all }x\in\rr.$$
It is clear that $\varphi_1,\varphi_2\in C^1(\rr)$ and there exists $C>0$ such that
\begin{equation}\label{ppty:phi}
0\leq \varphi_i(x)\leq C(1+|x|)\,,\, |\varphi_i'(x)|\leq C\hbox{ and }x=\varphi_1(x)-\varphi_2(x)\hbox{ for all }x\in \rr \hbox{ and }i=1,2.
\end{equation}
Define $f_i:=\varphi_i\circ f$ for $i=1,2$. In particular, $f=f_1-f_2$. Let $g_1,g_2\in \homega$ be   solutions to 
\begin{equation}
 (- \Delta)^{\frac{\alpha}{2}} g_i- \left(\frac{\gamma}{|x|^{\alpha}}+ a(x) \right) g_i = f_i\text{ weakly in } \homega\label{eq:f:plus}
\end{equation}
for $i=1,2$. Since $f_1,f_2\geq 0$, Lemma \ref{lem:comp} yields $g_1,g_2\geq 0$. Also
$$\left( (- \Delta)^{\frac{\alpha}{2}} - \left(\frac{\gamma}{|x|^{\alpha}}+ a(x) \right) \right) (g-(g_1-g_2))=f-(f_1-f_2)=0.$$
It  follows from coercivity that  $g= g_1-g_2$. Assuming $g_1 \not\equiv 0$, it follows from Lemma 3.10 in \cite{AMPP.CZ} that there exists $K' >0 $ such that $g_1(x) \ge K' |x|^{-\bm}$ in $B_\delta(0) \setminus \{0\}$.

\medskip\noindent Since $g_1\in \homega$, it follows from \eqref{eq:f:plus} and Theorem 2.1 of Jin-Li-Xiong \cite{Li_JEMS} that $g_1\in C^{0,\tau}_{loc}(\Omega\setminus\{0\})$ for some $\tau>0$. Arguing as in the proof of Theorem \ref{th:asymp:ext}, we get that $g_1\in C^1(\Omega\setminus\{0\})$. Setting 
$$h(x):=\frac{f_1(x)}{g_1(x)}\hbox{ for }x\hbox{ close to }0,$$
we have that $h\in C^1(B_\delta(0))$. Now use (\ref{bnd:f:0}) and \eqref{ppty:phi} to get that 
$$f_1(x) \le C(1+|f(x)|) \le C |x|^{-\bp} = C |x|^{-\bm} |x|^{\alpha -(\bp-\bm)} |x|^{-\alpha} \le K_1  |x|^{-\alpha+ (\alpha -(\bp-\bm))} g_1(x).$$
Using the fact that $\gamma> \gamma_{crit}(\alpha)$
if and only if $\alpha -(\bp-\bm) > 0$, we get that $|h(x)|\leq C|x|^{\tau-\alpha}$ for $x\to 0$ where $\tau:=\alpha -(\bp-\bm)>0$. Therefore, we have that 
$$(- \Delta)^{\frac{\alpha}{2}}g_1 - \frac{\gamma+ O(|x|^{ (\alpha -(\bp-\bm))}) }{|x|^{\alpha}} g_1 = 0 \text{ weakly in } \homega,$$
with $g_1\geq 0$ and $g_1\not\equiv 0$. It then follows from Lemma \ref{lem:control:u} that there exists $c>0$ such that $c^{-1}\leq |x|^{\bm}g_1(x)\leq c $ for $x\in\Omega$, $x\neq 0$ close to $0$.  Arguing as in Claim 2 in the proof of Theorem \ref{th:asymp:ext}, we get that there exists $C>0$ such that 
\begin{equation}\label{bnd:u}
c^{-1}\leq |x|^{\bm}g_1(x)\leq c \hbox{ and }|x|^{\bm+1}|\nabla g_1(x)|\leq C\hbox{ for all }x\in B_\delta(0).
\end{equation}

\medskip\noindent We now deal with the differential of $h$. With the controls \eqref{bnd:f:0}, \eqref{ppty:phi} and \eqref{bnd:u}, we get that
$$|x|\cdot|\nabla h(x)|\leq C|x|^{\tau-\alpha}\hbox{ for }x\in B_{\delta/2}(0)\setminus\{0\}.$$
Now, writing $(-\Delta)^{\frac{\alpha}{2}}g_1-\frac{\gamma}{|x|^\alpha}g_1=h(x)g_1$ in $\Omega$ and using Lemma \ref{lem:ext:fall:felli}, we get that $|x|^{-\bm}g_1(x)$ has a finite limit as $x \to 0$. Note that this is also clearly the case if $g_1\equiv 0$. The same holds for $g_2$. Therefore, there exists a constant $c \in \R$ such that 
$ |x|^{-\bm} g(x) \to c $ as $x \to 0. $
In other words,
$$H(x) = \frac{1}{|x|^{\bp}}+ \frac{c}{|x|^{\bm}}+ o \left( \frac{1}{|x|^{\bm}} \right) \quad \text{ as } x \to 0,$$
and there exists $C>0$ such that $|g(x)|\leq C|x|^{-\bm}$ for all $x\in \Omega$.\par

\medskip\noindent We now prove that $H>0$ in $\Omega\setminus\{0\}$. Indeed, from the above asymptotic expansion we have that $H(x)>0$ for $x\to 0$, $x\neq 0$. Since $\chi H\in \h$ for all $\chi\in C^\infty_c(\rn\setminus\{0\})$, it follows from Proposition \ref{prop:carac:h} that $H_-\in H_0^{\frac{\alpha}{2}}(\Omega\setminus B_\eps(0))$ for some $\eps>0$ small. We then test \eqref{eq:u:mass} against $H_-$ and, arguing as in the proof of Lemma \ref{lem:comp} we get that $H_-\equiv 0$, and then $H\geq 0$. Since $H\not\equiv 0$, $H\in C^1(\Omega\setminus\{0\})$, it follows from the Harnack inequality (see Lemma 3.10 in \cite{AMPP.CZ}) that $H>0$ in $\Omega\setminus\{0\}$. This proves the existence of a solution $u$ to Problem \eqref{Main problem:TheMass} with the relevant asymptotic behavior.

\medskip\noindent We now deal with uniqueness. Assume that there exists another solution $u'$ satisfying the hypothesis of Theorem \ref{Theorem:TheMass}. We define $\bar{u}:=u-u'$. Then $\bar{u}:\Omega\to\rr$ is a solution to
\begin{equation*}
\left\{\begin{array}{rl}
(-\Delta)^{\frac{\alpha}{2}}\bar{u}-\left( \frac{\gamma}{|x|^\alpha}+a(x)\right) \bar{u}=0 & \text{in }  {\Omega \setminus \{0\}}\\
 \bar{u}=0 & \text{in } \R^n \setminus \Omega,
\end{array}\right.
\end{equation*}
in the sense of Definition \ref{def:punct}. Since $|\bar{u}(x)|\leq C|x|^{-\bm}$ for all $x\in \Omega$ where $C>0$ is some uniform constant, then by using Proposition \ref{prop:carac:h} one concludes that $\bar{u}\in \homega$ is a weak solution to 
$$(-\Delta)^{\frac{\alpha}{2}}\bar{u}-\left( \frac{\gamma}{|x|^\alpha}+a(x)\right) \bar{u}=0  \hbox{ in }\Omega,
$$
that is,  for all $\varphi\in\homega$, 
$$\frac{C_{n,\alpha}}{2}\int_{\R^n} \int_{\R^n} \frac{(\bar{u}(x) - \bar{u}(y)) (\varphi(x) -\varphi(y))}{|x-y|^{n+\alpha}} dx dy - \int_{\rn} \left( \frac{\gamma}{|x|^\alpha}+a(x)\right)\bar{u}\varphi\,dx =0.
$$
Taking $\varphi:=\bar{u}$ and using the coercivity, we get that $\bar{u}\equiv 0$, and then $u\equiv u'$, which yields the uniqueness.\end{proof}

\section{Existence of extremals}

This section is devoted to prove the main result, which is Theorem \ref{th:1}. By choosing a  suitable test function, we  estimate the functional $J_a^\Omega(u)$, and  we show that the condition $\mu_{\gamma,s,\alpha,a}(\Omega) < \mu_{\gamma,s,\alpha}(\R^n)$ holds under suitable conditions on the dimension or on the mass of the domain. Recall that Proposition \ref{prop:exist} implies that it is this strict inequality that guarantees the existence of extremals for $\mu_{\gamma,s,\alpha,a}(\Omega)$.

\medskip\noindent We fix  $ a \in C^{0, \tau} (\overline{\Omega}),$ $\tau \in (0,1)$ and $\eta \in C_c^{\infty} (\Omega) $ such that
\begin{equation}\label{def:eta}
\eta\equiv 1\hbox{ in }B_\delta(0)\hbox{ and }\eta\equiv 0\hbox{ in }\rn\setminus B_{2\delta}(0)\hbox{ with }B_{4\delta}(0)\subset \Omega.
\end{equation}
Let $U\in \h$ be an extremal for $\mu_{\gamma,s,\alpha,0}(\rn)$. It follows from Theorem \ref{th:asymp:ext} that, up to multipliying by a nonzero constant, $U$ satisfies for some $\kappa>0$, 
\begin{equation}\label{eq:U}
(-\Delta)^{\frac{\alpha}{2}}U-\frac{\gamma}{|x|^\alpha}U=\kappa \frac{U^{\crits-1}}{|x|^s}\hbox{ weakly in }\h.
\end{equation}
 Moreover, $U\in C^1(\rn\setminus\{0\})$, $U>0$ and 
\begin{equation}\label{asymp:U}
\lim_{|x|\to \infty}|x|^{\bp}U(x)=1.
\end{equation}
Set
\begin{eqnarray}
J_a^\Omega(u):=\frac{\frac{C_{n,\alpha}}{2}\int_{(\rn)^2}\frac{(u(x)-u(y))^2}{|x-y|^{n+\alpha}}\, dxdy-\int_{\Omega}\left(\frac{\gamma}{|x|^\alpha}+a\right) u^2\, dx}{\left(\int_{\Omega}\frac{|u|^{\crits}}{|x|^s}\, dx\right)^{\frac{2}{\crits}}}=\frac{A(u)}{B(u)^{\frac{2}{\crits}}},\label{def:J:A:B}
\end{eqnarray}
where
\begin{equation}\label{def:A:B}
A(u):=\langle u, u\rangle-\int_\Omega au^2\, dx \quad \hbox{ and } \quad B(u):=\int_{\Omega}\frac{|u|^{\crits}}{|x|^s}\, dx 
\end{equation}
with
\begin{eqnarray}\label{def:bracket}
\langle u,v\rangle&:=&\frac{C_{n,\alpha}}{2}\int_{(\rn)^2}\frac{(u(x)-u(y))(v(x)-v(y))}{|x-y|^{n+\alpha}}\, dxdy\\
&&-\int_{\rn}\frac{\gamma}{|x|^\alpha}uv\, dx\hbox{ for }u,v\in \h.\nonumber
\end{eqnarray}
Consider
$$u_\epsilon(x):=\eps^{-\frac{n-\alpha}{2}}U(\eps^{-1}x)\hbox{ for }x\in\rn\setminus\{0\}.$$
It follows from Proposition \ref{prop:carac:h}, that $\eta\ue\in \homega$

\subsection{General estimates for $\eta\ue$}
We define the following bilinear form $B_\eta$ on $\h$ as follows:  For any $\varphi,\psi\in \h$,  \begin{eqnarray}
\qquad B_\eta(\varphi,\psi):=\langle \eta\varphi,\psi\rangle-\langle \varphi,\eta\psi\rangle= \frac{C_{n,\alpha}}{2}\int_{(\rn)^2}\frac{\eta(x)-\eta(y)}{|x-y|^{n+\alpha}}\left(\varphi(y)\psi(x)-\varphi(x)\psi(y)\right)\, dxdy. \label{def:B:eta}
\end{eqnarray} 
This expression makes sense since $\eta\equiv 1$ around $0$ and $\eta\equiv 0$ around $\infty$.  Note that 
\begin{eqnarray}\label{eq:60}
\langle \eta\ue,\eta\ue\rangle&=& \langle \ue, \eta^2\ue\rangle+\eps^{\bp-\bm}B_\eta\left(\frac{\ue}{\eps^{\frac{\bp-\bm}{2}}},\frac{\eta\ue}{\eps^{\frac{\bp-\bm}{2}}}\right).
\end{eqnarray}
It follows from \eqref{eq:U} and the definition of $\ue$ that
$$\langle \ue,\varphi\rangle=\kappa\int_{\rn}\frac{\ue^{\crits-1}}{|x|^s}\varphi\, dx\hbox{ for all }\varphi\in \h.$$
By a change of variable, we get as $\eps\to 0$,
\begin{eqnarray*}
\langle \ue,\eta^2\ue\rangle&=& \kappa\int_{\rn}\frac{\eta^2\ue^{\crits}}{|x|^s}\, dx=\kappa\int_{\rn}\frac{\ue^{\crits}}{|x|^s}\, dx+O\left(\int_{\rn\setminus B_\delta(0)}\frac{\ue^{\crits}}{|x|^s}\, dx\right)\\
&=&\kappa\int_{\rn}\frac{U^{\crits}}{|x|^s}\, dx+O\left(\int_{\rn\setminus B_{\eps^{-1}\delta}(0)}\frac{U^{\crits}}{|x|^s}\, dx\right).
\end{eqnarray*}
 With \eqref{asymp:U}, we get that
\begin{equation}\label{est:0}
\langle \ue,\eta^2\ue\rangle=\kappa\int_{\rn}\frac{U^{\crits}}{|x|^s}\, dx+O\left(\eps^{\frac{\crits}{2}(\bp-\bm)}\right)=\kappa\int_{\rn}\frac{U^{\crits}}{|x|^s}\, dx+o\left(\eps^{\bp-\bm}\right).
\end{equation}
We now deal with the second term of \eqref{eq:60}. First note that 
\begin{equation*}
B_\eta\left(\frac{\ue}{\eps^{\frac{\bp-\bm}{2}}},\frac{\eta\ue}{\eps^{\frac{\bp-\bm}{2}}}\right)=\frac{C_{n,\alpha}}{2}\int_{(\rn)^2}\frac{(\eta(x)-\eta(y))^2}{|x-y|^{n+\alpha}}\frac{\ue(x)}{\eps^{\frac{\bp-\bm}{2}}}\cdot \frac{\ue(y)}{\eps^{\frac{\bp-\bm}{2}}}\, dxdy.
\end{equation*}
It follows from \eqref{asymp:U} and the pointwise control of Theorem \ref{th:asymp:ext} that there exists $C>0$ such that for any $x\in\rn\setminus\{0\}$, we have that
\begin{equation}\label{est:tue}
\lim_{\eps\to 0}\frac{\ue(x)}{\eps^{\frac{\bp-\bm}{2}}}=S(x):=\frac{1}{|x|^{\bp}}\hbox{ and }\left|\frac{\ue(x)}{\eps^{\frac{\bp-\bm}{2}}}\right|\leq \frac{C}{|x|^{\bp}}.
\end{equation}
Since $\eta(x)=1$ for all $x\in B_\delta(0)$ and $\Bp<n$, Lebesgue's convergence theorem yields
\begin{equation}\label{lim:B:2:bis}
\lim_{\eps\to 0}B_\eta\left(\frac{\ue}{\eps^{\frac{\bp-\bm}{2}}},\frac{\eta\ue}{\eps^{\frac{\bp-\bm}{2}}}\right)=\frac{C_{n,\alpha}}{2}\int_{(\rn)^2}\frac{(\eta(x)-\eta(y))^2}{|x-y|^{n+\alpha}}S(x)S(y)\, dxdy=B_\eta(S,\eta S).
\end{equation}
By plugging together \eqref{eq:60}, \eqref{est:0} and \eqref{lim:B:2:bis}, we get as $\eps\to 0$, 
\begin{equation}\label{est:num:1}
\langle \eta\ue,\eta\ue\rangle=\kappa\int_{\rn}\frac{U^{\crits}}{|x|^s}\, dx+B_\eta(S,\eta S)\eps^{\bp-\bm}+o\left(\eps^{\bp-\bm}\right).
\end{equation}
Arguing as in the proof of \eqref{est:0}, we obtain as $\eps\to 0$, 
\begin{equation}\label{est:num:2}
\int_{\rn}\frac{(\eta \ue)^{\crits}}{|x|^s}\, dx=\int_{\rn}\frac{U^{\crits}}{|x|^s}\, dx+o(\eps^{\bp-\bm}).
\end{equation}
As an immediate consequence, we get 
\begin{proposition}\label{prop:val:best:cst}
Suppose that $0 \le s < \alpha <n,$ $0 < \alpha < 2$ and $ 0 \le \gamma < \gamma_H(\alpha).$ Then,  $$\mu_{\gamma,s,\alpha,0}(\Omega) = \mu_{\gamma,s,\alpha}(\R^n).$$
\end{proposition}

\begin{proof}

It follows from the definition of $\mu_{\gamma,s,\alpha} (\Omega)$ that $\mu_{\gamma,s,\alpha,0} (\Omega) \ge \mu_{\gamma,s,\alpha} (\R^n).$ We now show the reverse inequity. Using the estimates \eqref{est:num:1} and \eqref{est:num:2} above, we have as $\eps\to 0$, 
\begin{eqnarray*}
J_0^\Omega (U_\epsilon) &=&\frac{\kappa\int_{\rn}\frac{U^{\crits}}{|x|^s}\, dx}{\left(\int_{\rn}\frac{U^{\crits}}{|x|^s}\, dx\right)^{\frac{2}{\crits}}}+O(\eps^{\bp-\bm})\\
&=&J_0^{\rn}(U)+O(\eps^{\bp-\bm}) =  \mu_{\gamma,s,\alpha} (\R^n) + O(\eps^{\bp-\bm}).
\end{eqnarray*} 
Letting $ {\epsilon \to 0}$ yields  $\mu_{\gamma,s,\alpha,0} (\Omega) \le \mu_{\gamma,s,\alpha} (\R^n)$ from which follows that  $\mu_{\gamma,s,\alpha,0}(\Omega)= \mu_{\gamma,s,\alpha}(\R^n)$.\end{proof}

\subsection{Test functions for  the non-critical case $0 \le \gamma \le \gamma_{crit}(\alpha)$.} 

We now estimate $J(\eta\ue)$ when $0\le\gamma\le\gamma_{crit}(\alpha)$, that is in the case when 
$\bm\geq \frac{n}{2}.$
Note that since $\bm+\bp=n-\alpha$, we have that 
$$\hbox{$\bp-\bm>\alpha$ when $\gamma<\gamma_{crit}(\alpha)$ \quad and \quad $\bp-\bm=\alpha$ if $\gamma=\gamma_{crit}(\alpha).$}$$ 

We start with the following:
\begin{proposition} \label{prop:test:subcrit}
Let $0 \le s <\alpha<2,$ $0 \le \gamma \le \gamma_{crit}(\alpha)$ and $n \ge 2\alpha $. Then, as $\epsilon \to 0$, 
\begin{equation*} 
 \int_\Omega a(\eta\ue)^2 dx=\left\{\begin{array}{rl}
\epsilon^\alpha \left(\int_{\R^n} U^2 dx  \right) a(0)+  o(\epsilon^\alpha) \,\,\,\,\,\,\,\,\,\,\,\, & \text{if }   0 \le \gamma < \gamma_{crit}(\alpha)\\
\omega_{n-1} a(0)\epsilon^\alpha \ln(\epsilon^{-1}) +  o(\epsilon^\alpha \ln \epsilon)  & \text{if } \gamma = \gamma_{crit}(\alpha). 
\end{array}\right.  \end{equation*}
\end{proposition}

\begin{proof}[Proof of Proposition \ref{prop:test:subcrit}]
We write 
\begin{align*}
\int_\Omega a(\eta\ue)^2 dx&=\int_{B_\delta} a u_\epsilon^2 dx  +\int_{B_{2 \delta} \setminus B_\delta}a (\eta u_\epsilon)^2 dx\\
& = \epsilon^\alpha \int_{B_{\epsilon^{-1}\delta}} a(\eps x)U^2 dx + O(\eps^{\bp-\bm}).
\end{align*}
Assume that $\gamma<\gamma_{crit}(\alpha)$. Since $\bp > \frac{n}{2}$ and $U\in C^1(\rn\setminus\{0\})$ satisfies \eqref{ass1}, we get that $U\in L^2(\rn)$ and therefore, Lebesgue's convergence theorem and the assumption $\Bp-\Bm>\alpha$ yield
\begin{align*}
\int_\Omega a(\eta\ue)^2\, dx&= \epsilon^\alpha \left(\int_{\R^n} U^2 dx  \right)a(0) +  o(\epsilon^\alpha) \quad \text{ as } \epsilon \to 0.
\end{align*}
If now $\gamma = \gamma_{crit}(\alpha),$ then $\lim_{|x|\to \infty}|x|^{\frac{n}{2}}U(x)=1$ and $\bp-\bm =\alpha.$ Therefore
\begin{align*}
\int_\Omega (\eta\ue)^2 dx&= \omega_{n-1} a(0)\epsilon^\alpha \ln(\epsilon^{-1}) +  o(\epsilon^\alpha\ln\eps) \quad \text{ as } \epsilon \to 0.
\end{align*}
This proves Proposition \ref{prop:test:subcrit}.\end{proof}

\medskip\noindent Plugging together \eqref{est:num:1}, \eqref{est:num:2} and Proposition \ref{prop:test:subcrit} then yields, as $\eps\to 0$, 
\begin{eqnarray}
J_a^\Omega(\eta\ue)&=&\frac{\kappa\int_{\rn}\frac{U^{\crits}}{|x|^s}\, dx}{\left(\int_{\rn}\frac{U^{\crits}}{|x|^s}\, dx\right)^{\frac{2}{\crits}}}-a(0)\frac{\int_{\rn}U^2\, dx}{\left(\int_{\rn}\frac{U^{\crits}}{|x|^s}\, dx\right)^{\frac{2}{\crits}}}\eps^\alpha+o(\eps^\alpha)\nonumber\\
&=&J_0^{\rn}(U)-a(0)\frac{\int_{\rn}U^2\, dx}{\left(\int_{\rn}\frac{U^{\crits}}{|x|^s}\, dx\right)^{\frac{2}{\crits}}}\eps^\alpha+o(\eps^\alpha),\label{dev:souscrit}
\end{eqnarray}
when $\gamma<\gamma_{crit}(\alpha)$, and
\begin{eqnarray}
J_a^\Omega(\eta\ue)&=&J_0^{\rn}(U)-a(0)\frac{\omega_{n-1}}{\left(\int_{\rn}\frac{U^{\crits}}{|x|^s}\, dx\right)^{\frac{2}{\crits}}}\eps^\alpha\ln\frac{1}{\eps}+o(\eps^\alpha\ln\frac{1}{\eps}),\label{dev:crit}
\end{eqnarray}
when $\gamma=\gamma_{crit}(\alpha)$.

\subsection{The test function for the critical case}\label{subsec:test:subcrit} Here, we assume that $\gamma>\gamma_{crit}(\alpha)$. It follows from Theorem \ref{Theorem:TheMass} that there exists $H: \Omega\setminus\{0\}\to\rr$ such that 
\begin{equation*}
\left\{\begin{array}{ll}
H\in C^1(\Omega\setminus\{0\})\;,\;\xi H\in \homega &\hbox{ for all }\xi\in C_c^\infty(\rn\setminus\{0\}),\\
(-\Delta)^{\frac{\alpha}{2}}H-\left(\frac{\gamma}{|x|^\alpha}+a)\right)H=0& \hbox{ weakly in }\Omega\setminus\{0\}\\
H>0&\hbox{ in }\Omega\setminus\{0\}\\
H=0&\hbox{ in }\partial\Omega\\
{\rm and} \,\, \lim_{x\to 0}|x|^{\bp}H(x)=1.&
\end{array}\right.
\end{equation*}
Here the solution is in the sense of Definition \ref{def:punct}. In other words, the second identity means that for any $\varphi\in C^\infty_c(\Omega\setminus\{0\})$, we have that
\begin{equation}\label{eq:H}
\frac{C_{n,\alpha}}{2}\int_{(\rn)^2}\frac{(H(x)-H(y))(\varphi(x)-\varphi(y))}{|x-y|^{n+\alpha}}\, dxdy-\int_{\rn}\left(\frac{\gamma}{|x|^\alpha}+a\right) H\varphi\, dx=0.
\end{equation}
Note that this latest identity makes sense since $H\in L^1(\Omega)$ (since $\bp<n$). Let now $\eta$ be as in \eqref{def:eta}. Following the construction of the singular function $H$ in \eqref{def:u:mass}, there exists $g\in \homega$ such that 
$$H(x) := \frac{\eta(x)}{|x|^{\bp}} + g(x)\hbox{ for }x\in \Omega\setminus\{0\},$$
where 
\begin{equation}\label{eq:t}
(-\Delta)^{\frac{\alpha}{2}}g-\left(\frac{\gamma}{|x|^\alpha}+a\right)g=f,
\end{equation}
with $f\in L^\infty(\Omega)$ and $f\in C^1(B_\delta(0)\setminus\{0\})$. It follows from \eqref{bnd:f:0} that there  exists $c>0$ such that 
\begin{equation}\label{ppty:f}
|f(x)|\leq C|x|^{-\bp}\hbox{ for }x\in \Omega\setminus\{0\}\hbox{ and }|\nabla f(x)|\leq C|x|^{-\bp-1}\hbox{ for all }x\in B_{\delta/2}(0)\setminus\{0\}.
\end{equation}
We also have that
\begin{equation}\label{def:mass}
g (x)= \frac{m^\alpha_{\gamma,a}(\Omega)}{|x|^{\bm}}+o\left(\frac{1}{|x|^{\bm}}\right)\hbox{ as }x\to 0,\hbox{ and } |g(x)|\leq C|x|^{-\bm}\hbox{ for all }x\in\Omega.
\end{equation}
Define the test function as 
$$ T_\epsilon(x) = \eta u_\epsilon(x) + \epsilon^{\frac{\bp-\bm}{2}} g(x) \quad\text{ for all }  x \in \overline{\Omega} \setminus \{0\},$$
where 
$$u_\epsilon(x):=\eps^{-\frac{n-\alpha}{2}}U(\eps^{-1}x)\hbox{ for }x\in\rn\setminus\{0\},$$
and $U\in \h$ is such that $U>0$, $U\in C^1(\rn\setminus\{0\})$ and satisfies \eqref{eq:U} above for some $\kappa>0$ and also \eqref{asymp:U}. It is easy to see that $ \Te \in \homega$ for all  $\epsilon >0$. 

This subsection is devoted to computing the expansion of $J_a^\Omega(\Te)$ where $J_a^\Omega$ is defined in \eqref{def:J:A:B}, \eqref{def:A:B} and \eqref{def:bracket}. For simplicity, we set $S(x):=\frac{1}{|x|^{\bp}}$ for $x\in\rn\setminus \{0\}$. In particular, it follows from \eqref{lap:x:theta:dist} that we have that
\begin{equation}\label{eq:S}
(-\Delta)^{\frac{\alpha}{2}}S-\frac{\gamma}{|x|^\alpha}S=0\hbox{ weakly in }\rn\setminus\{0\},
\end{equation}
in the sense that $\langle S,\varphi\rangle=0$ for all $\varphi\in C^\infty_c(\rn\setminus\{0\})$. First note that 
\begin{equation*}
\lim_{\eps\to 0}\frac{T_\eps}{\eps^{\frac{\bp-\bm}{2}}}=H\hbox{ in }L^\infty_{loc}(\overline{\Omega}\setminus\{0\}).
\end{equation*}  
Therefore, since $|\eps^{-\frac{\bp-\bm}{2}}T_\eps(x)|\leq C|x|^{-\bp}$ for $x\in \Omega\setminus\{0\}$ with $2\bp<n$, Lebesgue's theorem yields as $\eps\to 0$, 
$$\int_\Omega a\Te^2\, dx=\eps^{\bp-\bm}\int_\Omega aH^2\, dx+ o\left(\eps^{\bp-\bm}\right),$$
Since $\Te=\eta\ue+\eps^{\frac{\bp-\bm}{2}}g$, we have that
\begin{eqnarray*}
A(\Te)&=& \langle \Te,\Te\rangle-\eps^{\bp-\bm}\int_\Omega aH^2\, dx+ o\left(\eps^{\bp-\bm}\right)\\
&=& \langle \eta\ue, \eta\ue\rangle+2 \eps^{\frac{\bp-\bm}{2}} \langle \eta\ue ,g \rangle + \eps^{\bp-\bm} \langle g, g\rangle-\eps^{\bp-\bm}\int_\Omega aH^2\, dx+ o\left(\eps^{\bp-\bm}\right) 
\end{eqnarray*}
We are now going to estimate these terms separately. First, Formula \eqref{def:B:eta} and \eqref{est:num:1} yield, as $\eps\to 0$ 
\begin{eqnarray}
A(\Te)&=& \kappa\int_{\rn}\frac{U^{\crits}}{|x|^s}\, dx +2 \eps^{\frac{\bp-\bm}{2}} \langle \ue ,\eta g \rangle +\eps^{\bp-\bm}M_\eps+ o\left(\eps^{\bp-\bm}\right), \label{est:A:1}
\end{eqnarray}
where
\begin{equation*}
M_\eps:=B_\eta\left(S,\eta S\right)+ 2  B_\eta\left(\frac{\ue}{\eps^{\frac{\bp-\bm}{2}}},g\right)+ \langle g, g\rangle-\int_\Omega aH^2\, dx.
\end{equation*}
As to the second term of \eqref{est:A:1}, we have 
$$\langle \ue ,\eta g \rangle=\kappa\int_{\rn}\frac{\ue^{\crits-1}\eta g}{|x|^s}\, dx.$$
We set $\theta_\eps:=\int_{\rn}\frac{\ue^{\crits-1}\eta g}{|x|^s}\, dx$. It is easy to check that, since $\eta g\in \homega$, $(\ue)_\eps$ is bounded in $\h$ and goes to $0$ weakly as $\eps\to 0$, we have that
\begin{equation}\label{lim:theta}
\lim_{\eps\to 0}\theta_\eps=0.
\end{equation}
Therefore we can rewrite \eqref{est:A:1} as
\begin{eqnarray}
A(\Te)&=& \kappa\int_{\rn}\frac{U^{\crits}}{|x|^s}\, dx+2 \kappa \eps^{\frac{\bp-\bm}{2}} \theta_\eps+\eps^{\bp-\bm}M_\eps+ o\left(\eps^{\bp-\bm}\right)\label{est:A:2}
\end{eqnarray}
as $\eps\to 0$.\\
We now estimate $M_\eps$. First, we write
$$B_\eta\left(\frac{\ue}{\eps^{\frac{\bp-\bm}{2}}},g\right)=\frac{C_{n,\alpha}}{2}\int_{(\rn)^2}F_\eps(x,y)\, dxdy,$$
where
$$F_\eps(x,y):=\frac{\eta(x)-\eta(y)}{|x-y|^{n+\alpha}}\left(\frac{\ue}{\eps^{\frac{\bp-\bm}{2}}}(y)g(x)-\frac{\ue}{\eps^{\frac{\bp-\bm}{2}}}(x)g(y)\right).$$
Remembering that $\eta\equiv 1$ in $B_\delta(0)$ and $\eta\equiv 0$ in $B_{2\delta}(0)^c$ and using \eqref{def:mass}, we get that
$$\left|F_\eps(x,y)\mathbbm{1}_{|x|<\delta/2}\right|\leq C\mathbbm{1}_{|x|<\delta/2}\mathbbm{1}_{|y|>\delta}|x|^{-\bp}|y|^{-(n+\alpha+\bm)}\in L^1((\rn)^2).$$
Similarly, we have a bound on $F_\eps$ on $\{|x|>3\delta\}$. By symmetry, this yields also a bound on $\{|y|<\delta/2\}\cup\{|y|>3\delta\}$. We are then left with getting a bound on ${\mathcal A}:=\left[B_{3\delta}(0)\setminus B_{\delta/2}(0)\right]^2$. 

For $(x,y)\in {\mathcal A}$, we have that
\begin{eqnarray*}
|F_\eps(x,y)|&\leq &C\frac{|x-y|}{|x-y|^{n+\alpha}}\cdot\left|\left(\frac{\ue}{\eps^{\frac{\bp-\bm}{2}}}(y)-\frac{\ue}{\eps^{\frac{\bp-\bm}{2}}}(x)\right)g(x)+\frac{\ue}{\eps^{\frac{\bp-\bm}{2}}}(x)(g(x)-g(y))\right|\\
&\leq &C|x-y|^{1-\alpha-n}\left(\left|\frac{\ue}{\eps^{\frac{\bp-\bm}{2}}}(y)-\frac{\ue}{\eps^{\frac{\bp-\bm}{2}}}(x)\right|+|g(x)-g(y)|\right).
\end{eqnarray*}
As noticed in the proof of Theorem \ref{Theorem:TheMass}, it follows from elliptic theory that $g\in C^1(\Omega\setminus\{0\})$. Therefore, there exists $C>0$ such that $|g(x)-g(y)|\leq C|x-y|$ for all $(x,y)\in {\mathcal A}$. 

Setting $\tilde{u}_\eps:=\eps^{-\frac{\bp-\bm}{2}}\ue$,  it follows from  \eqref{eq:U} that
$$(-\Delta)^{\frac{\alpha}{2}}\tilde{u}_\eps-\frac{\gamma}{|x|^\alpha}\tilde{u}_\eps=\kappa\eps^{\frac{\crits-2}{2}(\bp-\bm)}\frac{\tilde{u}_\eps^{\crits-1}}{|x|^s}\hbox{ weakly in }\h.$$
It then follows from \eqref{est:tue} and arguments similar to the Proof of  Theorem \ref{th:asymp:ext} (see Remark 2.5 and Theorem 2.1 of Jian-Li-Xiong \cite{Li_JEMS}) that $(\tilde{u}_\eps)_\eps$ is bounded in $C^1_{loc}(\rn\setminus\{0\})$. Therefore, there exists $C>0$ such that $|\tilde{u}_\eps(x)-\tilde{u}_\eps(y)|\leq C|x-y|$ for all $(x,y)\in {\mathcal A}$. Then, we get 
$$|F_\eps(x,y)|\leq C|x-y|^{2-\alpha-n}\in L^1({\mathcal A}).$$
Therefore, $(F_\eps)$ is uniformly dominated on $(\rn)^2$. Noting that $\frac{\ue}{\eps^{\frac{\bp-\bm}{2}}}(x)\to S(x)$ as $\eps\to 0$ for all $x\in \rn\setminus\{0\}$, Lebesgue's theorem yields
\begin{equation}\label{lim:B:1}
\lim_{\eps\to 0}B_\eta\left(\frac{\ue}{\eps^{\frac{\bp-\bm}{2}}},g\right)=B_\eta\left(S,g\right).
\end{equation}
Here again, note that $B_\eta(S,g)$ makes sense. Therefore, we get that $M_\eps=M+o(1)$ as $\eps\to 0$ where 
\begin{equation}\label{def:M}
M:=B_\eta(S,\eta S)+ 2  B_\eta\left(S,g\right)+ \langle g, g\rangle-\int_\Omega aH^2\, dx.
\end{equation}
We now estimate $B(\Te)$. Note first that since $p>2$, there exists $C(p)>0$ such that 
$$\left| |x+y|^p-|x|^p-p|x|^{p-2}xy \right|\leq C(p)\left(|x|^{p-2}y^2+|y|^p\right) \hbox{ for all $x,y\in\rr$}.$$
We therefore get that
\begin{eqnarray*}
B(\Te)&=& \int_{\rn}\frac{\left|\eta \ue+\eps^{\frac{\bp-\bm}{2}}g\right|^{\crits}}{|x|^s}\, dx=\int_{\rn}\frac{(\eta \ue)^{\crits}}{|x|^s}\, dx + \crits \eps^{\frac{\bp-\bm}{2}} \int_{\rn}\frac{\ue^{\crits-1}\eta^{\crits-1}g}{|x|^s}\, dx\\
&&+O\left(\eps^{\bp-\bm} \int_{\rn}\frac{\ue^{\crits-2}\eta^{\crits-2}g^2}{|x|^s}\, dx+\eps^{\frac{\crits}{2}(\bp-\bm)}\int_{\rn}\frac{|g|^{\crits}}{|x|^s}\, dx \right).\end{eqnarray*}
Since $\eta\equiv 1$ around $0$, we get that 
$$\int_{\rn}\frac{\ue^{\crits-1}\eta^{\crits-1}g}{|x|^s}\, dx=\int_{\rn}\frac{\ue^{\crits-1}\eta g}{|x|^s}\, dx+ O\left(\int_{\Omega\setminus B_{\delta}(0)}\frac{\ue^{\crits-1}g }{|x|^s}\, dx\right)= \theta_\eps+o\left(\eps^{\frac{\bp-\bm}{2}}\right)$$
as $\eps\to 0$. Therefore, in view of \eqref{est:num:2}, we deduce that
\begin{equation}\label{est:B}
B(\Te)=\int_{\rn}\frac{U^{\crits}}{|x|^s}\, dx+\crits \eps^{\frac{\bp-\bm}{2}}\theta_\eps+ o\left(\eps^{\bp-\bm}\right)
\end{equation}
as $\eps\to 0$. Plugging \eqref{est:A:2}, \eqref{lim:theta} and \eqref{est:B} into \eqref{def:J:A:B}, we get that
\begin{eqnarray}
J_a^\Omega(\Te)&=&\frac{\kappa \int_{\rn}\frac{U^{\crits}}{|x|^s}\, dx }{\left(\int_{\rn}\frac{U^{\crits}}{|x|^s}\, dx\right)^{\frac{2}{\crits}}}\left(1+\frac{M}{\kappa \int_{\rn}\frac{U^{\crits}}{|x|^s}\, dx}\eps^{\bp-\bm}+o\left(\eps^{\bp-\bm}\right)\right)\nonumber\\
&=&J_0^{\rn}(U)\left(1+\frac{M}{\kappa \int_{\rn}\frac{U^{\crits}}{|x|^s}\, dx}\eps^{\bp-\bm}+o\left(\eps^{\bp-\bm}\right)\right)\label{dev:JTE:1}
\end{eqnarray}
as $\eps\to 0$, where $M$ is defined in \eqref{def:M} and $J_0^{\rn}$ is as in \eqref{def:J:A:B}.

We now express $M$ in term of the mass. Note that in the classical (pointwise) setting, an integration by parts yield that $B_\eta(\varphi,\psi)$ defined in \eqref{def:B:eta} is an integral on the boundary of a domain. Hence, the mass appears by simply integrating by part independently the singular function $H$. The central remark we make here is that the integral on the boundary on a domain (defined in the local setting) can be seen as the limit of an integral on the domain via multiplication by a cut-off function with support converging to the boundary --which happened to be defined in the nonlocal setting. Therefore, despite the nonlocal aspect of our problem, we shall be able to apply the same strategy as in the local setting. 

\smallskip\noindent We shall be performing the following computations in the same order as the ones above made to get $A(\Te)$. The constant $M$ will therefore appear naturally in the two settings.

\medskip\noindent Let $\chi\in C^\infty(\rn)$ such that $\chi\equiv 0$ in $B_1(0)$ and $\chi\equiv 1$ in $\rn\setminus B_{2}(0)$. For $k\in\nn\setminus\{0\}$, define $\chi_k(x):=\chi(kx)$ for $x\in\rn$, so that
$$\chi_k(x)=0\hbox{ for }|x|<\frac{1}{k}\hbox{ and }\chi_k(x)=1\hbox{ for }|x|>\frac{2}{k}.$$
In particular, $(\chi_k)_k$ is bounded in $L^\infty(\rn)$ and $\chi_k(x)\to 1$ as $k\to +\infty$ for a.e. $x\in\rn$. Since $\chi_k H\in \homega$, then by the very definition of $H$ (see \eqref{eq:H}), we have that
\begin{eqnarray*}
0&=& \langle H,\chi_k H\rangle-\int_{\rn}aH \chi_k H\, dx\\
&=& \langle \eta S+g,\chi_k \eta S+\chi_k g \rangle-\int_{\rn}\chi_k aH^2\, dx\\
&=& \langle \eta S,\chi_k \eta S\rangle+\langle \eta S , \chi_k g\rangle+\langle \chi_k\eta S, g\rangle+\langle g,\chi_k g \rangle
-\int_{\rn}\chi_k aH^2\, dx\\
&=& \langle  S,\chi_k \eta^2 S\rangle+B_\eta(S,\chi_k \eta S)+\langle S , \eta \chi_k g\rangle+B_\eta(S,\chi_k g)+\langle  S, \chi_k\eta g\rangle\\
&&+B_{\chi_k\eta}(S,g)+\langle g,\chi_k g \rangle
-\int_{\rn}\chi_k aH^2\, dx.
\end{eqnarray*}
Since $aH^2\in L^1(\Omega)$ (this is a consequence of $2\bp<n$) and $S$ is a solution to \eqref{eq:S}, we get that
\begin{eqnarray*}
0&=&B_\eta(S,\chi_k \eta S)+B_\eta(S,\chi_k g)+B_{\chi_k\eta}(S,g)+\langle g,\chi_k g \rangle-\int_{\rn} aH^2\, dx+o(1)
\end{eqnarray*}
as $k\to +\infty$. We now estimate these terms separately. \\
Our first claim is that 
\begin{equation}\label{lim:tt}
\lim_{k\to +\infty}\langle \chi_k g,g\rangle=\langle g,g\rangle.
\end{equation}
Indeed,
\begin{eqnarray*}
\Vert \chi_k g-g\Vert_{\h}^2&=& \frac{C_{n,\alpha}}{2}\int_{(\rn)^2}\frac{\left|(1-\chi_k)(x)g(x)-(1-\chi_k)(y)g(y)\right|^2}{|x-y|^{n+\alpha}}\, dxdy\\
&\leq & C_{n,\alpha}\int_{(\rn)^2}|1-\chi_k(x)|^2\frac{|g(x)-g(y)|^2}{|x-y|^{n+\alpha}}\, dxdy\\
&&+C_{n,\alpha}\int_{(\rn)^2}g(y)^2\frac{|\chi_k(x)-\chi_k(y)|^2}{|x-y|^{n+\alpha}}\, dxdy.
\end{eqnarray*}
The first integral goes to $0$ as $k\to +\infty$ with Lebesgue's convergence theorme since $g\in \h$. For the second term, we use the change of variable $X=kx$, $Y=ky$ and the control of $g(x)$ by $|x|^{-\bm}$. This proves that $(\chi_k g)\to g$ in $\h$ as $k\to +\infty$. The claim follows and \eqref{lim:tt} is proved. 

\medskip\noindent We now write
\begin{eqnarray*}
B_\eta(S,\chi_k \eta S)&=&  \frac{C_{n,\alpha}}{2}\left(\int_{(\rn)^2}\chi_k(x)\tilde{F}(x,y)\, dxdy+\int_{(\rn)^2}G_k(x,y)\, dxdy\right),
\end{eqnarray*}
where 
$$\tilde{F}(x,y):=\frac{\eta(x)-\eta(y)}{|x-y|^{n+\alpha}}\left(S(y)(\eta S)(x)-S(x)(\eta S)(y)\right),$$
and
$$G_k(x,y):=\frac{(\eta(x)-\eta(y))(\chi_k(x)-\chi_k(y))(\eta S)(y) S(x)}{|x-y|^{n+\alpha}}.$$
As in the proof of \eqref{lim:B:1} and \eqref{lim:B:2:bis}, $\tilde{F}\in L^1((\rn)^2)$ and Lebesgue's convergence theorem yields
$$\lim_{k\to +\infty}\frac{C_{n,\alpha}}{2}\int_{(\rn)^2}\chi_k(x)\tilde{F}(x,y)\, dxdy=B_\eta(S,\eta S).$$
Arguing as in the proof of \eqref{lim:B:1}, we get the existence of $G\in L^1((\rn)^2)$ such that $|G_k(x,y)|\leq G(x,y)$ for all $(x,y)\in (\rn)^2$ such that $|x|<\delta/2$ or $|x|>3\delta$. By symmetry, a similar control also holds for $(x,y)\in (\rn)^2$ such that $|y|<\delta/2$ or $|y|>3\delta$. Moreover, for $\eps>0$ small enough, we have that $G_k(x,y)=0$ for $(x,y)\in (\rn)^2$ such that $|x|>\delta/2$ and $|y|>\delta/2$ (this is due to the definition of $\chi_k$). Therefore, since $\lim_{k\to +\infty}(\chi_k(x)-\chi_k(y))=0$ for a.e. $(x,y)\in(\rn)^2$, Lebesgue's convergence theorem yields $\int_{(\rn)^2}G_k(x,y)\, dxdy\to 0$ as $k\to +\infty$. We can then conclude that
$$\lim_{k\to +\infty}B_\eta(S,\chi_k \eta S)=B_\eta(S, \eta S).$$
Similar arguments yield 
$$\lim_{k\to +\infty}B_\eta(S,\chi_k g)=B_\eta(S, g).$$
Therefore, we get that
\begin{equation*}
0=B_\eta(S, \eta S)+B_\eta(S, g)+B_{\chi_k\eta}(S,g)+\langle g, g \rangle-\int_{\rn} aH^2\, dx+o(1)
\end{equation*}
as $k\to +\infty$. We also have that
\begin{eqnarray*}
B_{\chi_k \eta}(S,g)&=& \frac{C_{n,\alpha}}{2}\int_{(\rn)^2}\chi_k(x)\frac{\eta(x)-\eta(y)}{|x-y|^{n+\alpha}}\left(S(y)g(x)-S(x)g(y)\right)\, dxdy\\
&&+\frac{C_{n,\alpha}}{2}\int_{(\rn)^2}\eta(y)\frac{\chi_k(x)-\chi_k(y)}{|x-y|^{n+\alpha}}\left(S(y)g(x)-S(x)g(y)\right)\, dxdy.
\end{eqnarray*}
As above, the first integral of the right-hand-side goes to $B_\eta(S,g)$ as $k\to +\infty$. We now deal with the second integral. Using that $\bp+\bm=n-\alpha$, the change of variables $X=kx$ and $Y=ky$ yield
\begin{eqnarray*}
\int_{(\rn)^2}\eta(y)\frac{\chi_k(x)-\chi_k(y)}{|x-y|^{n+\alpha}}\left(S(y)g(x)-S(x)g(y)\right)\, dxdy
=\int_{(\rn)^2}F_k(X,Y)\, dXdY,
\end{eqnarray*}
where
$$F_k(X,Y):=\eta\left(\frac{Y}{k}\right)\frac{\chi(X)-\chi(Y)}{|X-Y|^{n+\alpha}}\left(\frac{1}{|Y|^{\bp}}g\left(\frac{X}{k}\right)k^{-\bm}-\frac{1}{|X|^{\bp}}g\left(\frac{Y}{k}\right)k^{-\bm}\right).$$
Note that there exists $C>0$ such that $|g(x)|\leq C|x|^{-\bm}$ for all $x\in\Omega\setminus\{0\}$. Since $\chi(X)=0$ for $|X|<1$ and $\chi(X)=1$ for $|X|>2$, arguing as in the proof of \eqref{lim:B:1}, we get that $|F_k(X,Y)|$ is uniformly bounded from above by a function in $L^1((\rn)^2)$ for $(X,Y)\in (\rn)^2$ such that $X\not\in B_3(0)\setminus B_{1/2}(0)$ or $Y\not\in B_3(0)\setminus B_{1/2}(0)$. 

\smallskip\noindent There exists $C>0$ such that $|\eta(X)-\eta(Y)|\leq C|X-Y|$ for all $(X,Y)\in [B_3(0)\setminus B_{1/2}(0)]^2$. Therefore, for such $(X,Y)$, we have that
\begin{eqnarray*}
|F_k(X,Y)|&\leq & C|X-Y|^{1-\alpha-n}\left| \left(\frac{1}{|Y|^{\bp}}-\frac{1}{|X|^{\bp}}\right)g\left(\frac{X}{k}\right)k^{-\bm}\right|\\
&&+C|X-Y|^{1-\alpha-n}\frac{1}{|X|^{\bp}}\left|g\left(\frac{X}{k}\right)k^{-\bm}-g\left(\frac{Y}{k}\right)k^{-\bm}\right|\\
&\leq & C|X-Y|^{2-\alpha-n}+ C|X-Y|^{1-\alpha-n}\left|g\left(\frac{X}{k}\right)k^{-\bm}-g\left(\frac{Y}{k}\right)k^{-\bm}\right|.
\end{eqnarray*}
Define $g_k(X):=g\left(\frac{X}{k}\right)k^{-\bm}$ for $X\in k\Omega$. It follows from \eqref{eq:t} and \eqref{ppty:f} that 
$$(-\Delta)^{\frac{\alpha}{2}}g_k-\left(\frac{\gamma}{|X|^\alpha}+k^{-\alpha}a(k^{-1}X)\right)g_k=f_k\hbox{ weakly in }H_0^{\alpha/2}(k\Omega)$$
where
$$f_k(X):=k^{-\bm-\alpha}f(k^{-1}X)\hbox{ so that }|f_k(X)|\leq Ck^{-(\alpha-(\bp-\bm))}|X|^{-\bp}$$
for all $X\in k\Omega$. Here again, elliptic regularity yields that $(g_k)$ is bounded in $C^1_{loc}(\rn\setminus\{0\})$. Therefore, there exists $C>0$ such that $$\left|g_k(X)-g_k(Y)\right|\leq C|X-Y|$$
for all $(X,Y)\in [B_3(0)\setminus B_{1/2}(0)]^2$. Therefore, we get that
$$|F_k(X,Y)|\leq C|X-Y|^{2-\alpha-n}\hbox{ for all }(X,Y)\in [B_3(0)\setminus B_{1/2}(0)]^2.$$
Therefore, since $\alpha<2$, $(F_k)$ is also dominated on this domain, and then on $(\rn)^2$. Finally, it follows from the definition \eqref{def:mass} of the mass that
$$\lim_{k\to +\infty}F_k(X,Y)=m_{\gamma,a}^\alpha(\Omega)\frac{\chi(X)-\chi(Y)}{|X-Y|^{n+\alpha}}\left(\frac{1}{|Y|^{\bp}|X|^{\bm}}-\frac{1}{|X|^{\bp}|Y|^{\bm}}\right)$$
for a.e. $(X,Y)\in (\rn)^2$. Therefore, Lebesgue's convergence theorem yields
$$0=B_\eta(S, \eta S)+2B_\eta(S, g)+K\cdot m_{\gamma,a}^\alpha(\Omega)+\langle g, g \rangle-\int_{\rn} aH^2\, dx,
$$
where $$K:=\int_{(\rn)^2}\frac{\chi(X)-\chi(Y)}{|X-Y|^{n+\alpha}}\left(\frac{1}{|Y|^{\bp}|X|^{\bm}}-\frac{1}{|X|^{\bp}|Y|^{\bm}}\right)\, dXdY.$$
Without loss of generality, we can assume that $\chi$ is radially symetrical and nondecreasing. Therefore, we get that $K>0$. With \eqref{def:M}, we then get that
$$M=-K\cdot m_{\gamma,a}^\alpha(\Omega)\hbox{ with }K>0.$$
Plugging this identity in \eqref{dev:JTE:1} yields
\begin{eqnarray}
J_a^\Omega(\Te)=J_0^{\rn}(U)\left(1-\frac{K}{\kappa \int_{\rn}\frac{U^{\crits}}{|x|^s}\, dx}\cdot m_{\gamma,a}^\alpha(\Omega) \eps^{\bp-\bm}+o\left(\eps^{\bp-\bm}\right)\right)\label{dev:JTE:2}
\end{eqnarray}

\subsection{Proof of Theorem \ref{th:1}} Theorem \ref{th:1} is now a direct consequence of \eqref{dev:souscrit}, \eqref{dev:crit}, \eqref{dev:JTE:2} and Proposition \ref{prop:exist}.

\end{document}